\documentclass[11pt]{amsart}
 \usepackage{amssymb,amsthm}
 \usepackage{amsmath,amssymb,amsfonts,amsthm} 
\usepackage{color}                    
\usepackage{subfigure}  
\usepackage{tikz}
\usepackage{graphics}                 
\usepackage{verbatim}   

\usetikzlibrary{arrows,chains,matrix,positioning,scopes}
\makeatletter
\tikzset{join/.code=\tikzset{after node path={%
\ifx\tikzchainprevious\pgfutil@empty\else(\tikzchainprevious)%
edge[every join]#1(\tikzchaincurrent)\fi}}}
\makeatother
\tikzset{>=stealth',every on chain/.append style={join},
         every join/.style={->}}
\tikzstyle{labeled}=[execute at begin node=$\scriptstyle,
   execute at end node=$]
\newtheorem{theorem}{Theorem}[section]
\newtheorem{proposition}[theorem]{Proposition}
\newtheorem{corollary}[theorem]{Corollary}
\newtheorem{lemma}[theorem]{Lemma}
\newtheorem{remark}[theorem]{Remark}
\newtheorem{definition}[theorem]{Definition}
\makeatletter
\newcommand{\dotminus}{\mathbin{\text{\@dotminus}}}

\newcommand{\@dotminus}{%
  \ooalign{\hidewidth\raise1ex\hbox{.}\hidewidth\cr$\m@th-$\cr}%
}
\makeatother

\date{\today}
\title{Reduced products of metric structures: a metric Feferman-Vaught theorem}

\author{Saeed Ghasemi}
\address{Institute of Mathematics, Polish Academy of Sciences,
ul. \'Sniadeckich 8,  00-656 Warszawa, Poland}
\email{\texttt{sghasemi@impan.pl}}

\begin{document}
\maketitle
\begin{abstract}
We extend the classical Feferman-Vaught theorem to  logic for metric structures. This implies that the reduced powers of elementarily equivalent structures are elementarily equivalent, and therefore they are isomorphic under the Continuum Hypothesis. We also prove the existence of two separable C*-algebras of the form $\bigoplus_i  M_{k(i)}(\mathbb{C})$ such that the assertion that their coronas are isomorphic is independent from ZFC, which gives the first example of genuinely non-commutative coronas of separable C*-algebras with this property.
\end{abstract}
\section{introduction}
In classical model theory S. Feferman and R.L. Vaught (\cite{F-V} and \cite[\S6.3]{Chang}) gave an effective (recursive) way to determine the satisfaction of formulas in the reduced products of models of the same language, over the ideal of all finite sets, $Fin$. They showed the preservation of the elementary equivalence relation $\equiv$ by arbitrary direct products and also by reduced products over $Fin$. Later Frayne, Morel and Scott (\cite{FMS}) noticed that the results extend to arbitrary reduced products (see also \cite{Yiannis}). Even though reduced products have been vastly studied for various metric structures, e.g., Banach spaces, C*-algebras, etc., unlike classical first order logic, their model theory has not been studied until very recently in \cite{Lopes} and \cite{FaShRig}.  The classical Feferman-Vaught theorem  effectively determines the truth value of a formula $\varphi$ in  reduced products of discrete structures $\{\mathcal{A}_{\gamma}: \gamma \in \Omega\}$ over an ideal $\mathcal{I}$ on $\Omega$, by the truth values of certain formulas in the models $\mathcal{A}_{\gamma}$ and in the Boolean algebra $P(\Omega)/\mathcal{I}$.

 In the present paper we prove a metric version of the Feferman-Vaught theorem (Theorem \ref{FV}) for reduced products of metric structures, which also implies the preservation of $\equiv$ by arbitrary direct products, ultraproducts and reduced products of metric structures. This answers a question stated in \cite{Lopes}.  We also use this theorem to solve an outstanding problem on coronas of C*-algebras (\S\ref{Matrices}).

 In the last few years the model theory for operator algebras has been developed and specialized from the model theory for metric structures. The reader may refer to \cite{FHS} for a detailed introduction to the model theory of operator algebras, and        \cite{FaICM} for an overview of the applications of logic to operator algebras. This  has proved to be very fruitful as many properties of C*-algebras and tracial von Neumann algebras have equivalent model theoretic reformulations (\cite{FHS}, \cite{FH}, $\dots$). In particular for a sequence of C*-algebras $\{\mathcal{A}_{n}: n\in\mathbb{N}\}$ the \emph{asymptotic sequence algebra} $\prod_{n}\mathcal{A}_{n}/\bigoplus_{n}\mathcal{A}_{n}$ is the reduced product over the Fr\'{e}chet ideal and is an important example of corona algebras. If $\mathcal A_n = \mathcal A$ for all $n $, we write $\ell_{\infty}(\mathcal{A})/c_{0}(\mathcal{A})$ instead of $\prod_{n}\mathcal{A}_{n}/\bigoplus_{n}\mathcal{A}_{n}$
In section \ref{corona-subsection} we briefly recall some of the basic concepts regarding corona algebras and the isomorphisms between them.

In model theory saturated structures have many intriguing properties. 
 A transfinite extension of Cantor's back-and-forth method shows that for any uncountable cardinal $\kappa$, two elementarily equivalent $\kappa$-saturated 
  metric structures (in the same countable language $\mathcal L$)  of density character $\kappa$, are isomorphic (see e.g., \cite{FHS} or \cite{Ben}). Recall that for a topological space $X$, the density character of $X$ is the smallest cardinality of a dense subset of $X$. 
   In \cite{FaShRig} Farah-Shelah (Theorem \ref{countable saturation}) showed that for any sequence  $\{\mathcal{A}_{n}: n\in\mathbb{N}\}$ of metric structures the reduced product  $\prod_n \mathcal A_n /\bigoplus_n \mathcal A_n$  over the Fr\'{e}chet ideal is  $\aleph_1$-saturated (we will usually say ``countably saturated'' instead of  $\aleph_1$-saturated). Hence the question whether two such reduced products of the density character $\aleph_1$ are isomorphic reduces to the weaker question of whether they are elementarily equivalent.
   For instance if $\mathcal A$ is separable C*-algebra, 
the asymptotic sequence algebra $\ell_{\infty}(\mathcal{A})/c_{0}(\mathcal{A})$  is  the corona algebra of the separable C*-algebra $\ell_{\infty}(\mathcal{A})$ and therefore it  is non-separable (see \cite[Theorem 2.7]{multiplier}). 
 Thus under the Continuum Hypothesis, the asymptotic sequence algebras of two separable C*-algebras $\mathcal A$ and $\mathcal B$ are isomorphic if and only if they  (the asymptotic sequence algebras) are elementarily equivalent.  One of the main consequences of our metric Feferman-Vaught theorem (Proposition \ref{preservation}) implies that if  separable $\mathcal A$ and $\mathcal B$ are elementarily equivalent, so are their asymptotic sequence algebras and therefore  $\ell_{\infty}(\mathcal{A})/c_{0}(\mathcal{A})$ is isomorphic to $\ell_{\infty}(\mathcal{B})/c_{0}(\mathcal{B})$, under the Continuum Hypothesis.

 We say an ideal $\mathcal{I}$ on $\mathbb{N}$ is \emph{atomless} if the Boolean algebra $P(\mathbb{N})/\mathcal{I}$ is atomless. The metric extension of the Feferman-Vaught theorem is used to prove the following theorem.
 \begin{theorem}\label{ee}
 Suppose $\mathcal{A}$ is a metric $\mathcal{L}$-structure and ideals $\mathcal{I}$ and $\mathcal{J}$ on $\mathbb{N}$ are atomless, then the reduced powers of $\mathcal{A}$ over $\mathcal{I}$ and $\mathcal{J}$ are elementarily equivalent.
\end{theorem}
Therefore in particular if $\mathcal{A}$ is a separable C*-algebra then under the Continuum Hypothesis such reduced powers of $\mathcal{A}$, if they are countably saturated, are all isomorphic to  $\ell_{\infty}(\mathcal{A})/c_{0}(\mathcal{A})$.

 For an ultrafilter $\mathcal{U}$ {\L}o\'{s}'s theorem implies that a metric structure $\mathcal{A}$ is elementarily equivalent to its ultrapower $\mathcal{A}_{\mathcal{U}}$. Therefore Farah-Shelah's result shows that under the Continuum Hypothesis if $\mathcal{A}$ is a separable C*-algebra,  $\ell_{\infty}(\mathcal{A})/c_{0}(\mathcal{A})$ is isomorphic to its ultrapower associated with any nonprincipal ultrafilter on $\mathbb{N}$ (\cite[Corollary 4.1]{FaShRig}).   Theorem \ref{ee} can be used (\S\ref{Last}) to show that under the Continuum Hypothesis any reduced power of an asymptotic sequence algebra $\ell_{\infty}(\mathcal{A})/c_{0}(\mathcal{A})$ over a large class of atomless  ideals is also isomorphic to $\ell_{\infty}(\mathcal{A})/c_{0}(\mathcal{A})$ itself.

 In Section \ref{Matrices} we show there are two reduced products (of matrix algebras) which are isomorphic under the Continuum Hypothesis but there are no ``trivial'' isomorphisms (see \ref{corona-subsection}) between them. Commutative examples of such reduced products are well-known, for example under  the Continuum Hypothesis $C(\beta\omega\setminus \omega)\cong C(\beta\omega^{2}\setminus \omega^{2})$ (note that $\ell_{\infty}/c_{0}\cong C(\beta\omega\setminus \omega)$), since by a well-known result of Parovi\v{c}enko (\cite{parov}) under  the Continuum Hypothesis $\beta\omega\setminus \omega$ and $\beta\omega^{2} \setminus \omega^{2}$ are homeomorphic.  However under the \emph{proper forcing axiom} they are not isomorphic (see \cite{DowHart} and \cite[Chapter 4]{FaAn}). A naive way to obtain nontrivial isomorphisms, under the Continuum Hypothesis, between non-commutative coronas is by tensoring $C(\beta\omega\setminus \omega)$ and $C(\beta\omega^{2}\setminus \omega^{2})$ with a full matrix algebra.  However, such nontrivial isomorphisms are just amplifications of the nontrivial isomorphisms between their corresponding commutative factors (see Section \ref{Matrices} for details). It has been asked by I. Farah to give examples of genuinely non-commutative corona algebras which are non-trivially isomorphic under the Continuum Hypothesis, for non-commutative reasons. Let $M_{n}$ denote the space of all $n\times n$ matrices over the field of complex numbers. 
 In \cite{Ghasemi} it has been proved that it is relatively consistent with ZFC that all isomorphisms between reduced products of matrix algebras over analytic P-ideals (e.g., the corona of $\bigoplus M_{k(n)}$) are trivial (cf. Theorem \ref{G}).
 
\begin{theorem}\label{main}
There is an increasing sequence of natural numbers $\{k_{\infty}(i): i\in\mathbb{N}\}$ such that if $\{g(i)\}$ and $\{h(i)\}$ are two subsequences of $\{k_{\infty}(i)\}$, then  under the Continuum Hypothesis, $\mathcal{M}_{g}=\prod_{i}M_{g(i)} /\bigoplus M_{g(i)}$ is isomorphic to $\mathcal{M}_{h}= \prod_{i} M_{h(i)}/\bigoplus M_{h_(i)}$. Moreover, the following are equivalent.
\begin{enumerate}
\item $\mathcal{M}_{g}$ and $\mathcal{M}_{h}$ are isomorphic in ZFC.\\
\item $\mathcal{M}_{g}$ and $\mathcal{M}_{h}$ are trivially isomorphic, i.e., $\{g(i): i\in \mathbb{N}\}$ and $\{h(i): i\in \mathbb{N}\}$ are equal modulo finite sets.
\end{enumerate}
\end{theorem}

  Thus if $\{g(i): i\in \mathbb{N}\}$ and $\{h(i): i\in \mathbb{N}\}$ are almost disjoint, this gives an example of two genuinely non-commutative reduced products for which the question ``whether or not they are isomorphic?", is independent from ZFC. We will also show (Theorem \ref{many}) that there is an abundance of different theories of reduced products of sequences of matrix algebras, by exhibiting $2^{\aleph_{0}}$ pairwise non-elementarily equivalent such reduced products.\\

\textbf{ACKNOWLEDGMENTS}. I would like to thank my Ph.D supervisor Ilijas Farah for eye-opening supervision and numerous comments and remarks. Many thanks to Vinicius Cif\'{u} Lopes, Seyed Mohammad Bagheri, Christopher Eagle and Martino Lupini for the comments, and to Bradd Hart for introducing me to some different aspects of the topic. I would also like to thank the referee for many suggestions and remarks.

\section{Some preliminaries}

\subsection{Corona algebras and trivial isomorphisms}\label{corona-subsection} The reader may refer to \cite{Black} for a detailed book on C*-algebras.  A  C*-algebra $\mathcal{A}$ is a Banach *-algebra over the field of complex numbers which satisfies the C*-identity
\begin{equation}
\nonumber \|xx^{*}\|=\|x\|^{2}
\end{equation}
for all $x\in\mathcal{A}$.
The Gelfand-Naimark theorem states that every commutative C*-algebra is isomorphic to some $C_0(X)$, the algebra of all continuous functions  on a locally compact Hausdorff space $X$ which vanish at infinity, with the pointwise multiplication. In general a C*-algebra may not have a unit element.
For a non-unital C*-algebra $\mathcal{A}$ there are various ways in which $\mathcal{A}$ can be embedded as an (essential) ideal in a unital C*-algebra. If $\mathcal A=C_{0}(X)$ is commutative this corresponds to the ways in which the locally compact Hausdorff space $X$ can be embedded as an open (dense) set in a compact Hausdorff space $Y$. The analogue of the \v{C}ech-Stone compactification for (non-commutative) C*-algebras is called the \emph{multiplier algebra} of $\mathcal{A}$.
The multiplier algebra $\mathcal{M}(\mathcal{A})$ of $\mathcal{A}$ is the unital C*-algebra containing $\mathcal{A}$ as an essential  ideal, which is universal in the sense that whenever $\mathcal{A}$ is an ideal of a unital C*-algebra $\mathcal{D}$, the identity map on $\mathcal{A}$ extends uniquely to a *-homomorphism from $\mathcal{D}$ into $\mathcal{M}(\mathcal{A})$ (cf. \cite[ II.7.3]{Black}).
The quotient C*-algebra $\mathcal{C}(\mathcal{A})=\mathcal{M}(\mathcal{A})/\mathcal{A}$ is called the \emph{corona} of $\mathcal{A}$.
A few well-known examples are corona algebras are listed below. 

(i) If $\mathcal{A}$ is unital, then $\mathcal{M}(\mathcal{A})=\mathcal{A}$. Therefore the corona of $\mathcal{A}$ is trivial.

(ii) If $\mathcal{A}= \mathcal{K}(H)$, the algebra of all compact operators on a Hilbert space $H$, then $\mathcal{M}(\mathcal{A})=\mathcal{B}(H)$, the algebra of all bounded linear operators on $H$. The corona of $\mathcal{A}$ is the Calkin algebra $\mathcal{C}(H)$.

(iii) If $X$ is a locally compact Hausdorff space, then  $\mathcal{M}(C_{0}(X))\cong C(\beta X)$, which is isomorphic to the C*-algebra $C_{b}(X)$ of bounded continuous complex-valued functions on $X$. The corona of $C_{0}(X)$ is isomorphic to the C*-algebra $C_{b}(X)/C_{0}(X)\cong C(X^{*})$, where $X^{*}$ is the \v{C}ech-Stone remainder $\beta X\setminus X$ of $X$.

(iv) Suppose $\mathcal{A}_{n}$ is a sequence of unital C*-algebras and let
\begin{eqnarray}
\nonumber &&\prod_{n}^{\infty} \mathcal{A}_{n}=\{(x_{n}): ~~x_{n}\in \mathcal{A}_{n}~~\text{and}~~ \sup_{n} \|x_{n}\|< \infty\}\\
\nonumber &&\bigoplus_{n}^{\infty} \mathcal{A}_{n}=\{(x_{n})\in \prod_{n} \mathcal{A}_{n}:  \|x_{n}\|\rightarrow 0\}.
\end{eqnarray}
If $\mathcal{A}= \bigoplus_{n} \mathcal{A}_{n}$ then $\mathcal{M}(\mathcal{A})=\prod_{n} \mathcal{A}_{n}$ and the corona of $\mathcal{A}$ is $\prod_{n} \mathcal{A}_{n}/\bigoplus_{n} \mathcal{A}_{n}$.

(v) If $\mathcal{A}$ is a C*-algebra and $X$ is a locally compact Hausdorff space, let
 \begin{eqnarray}
\nonumber &&C_{b}(X, \mathcal{A})=\{f: X\rightarrow \mathcal{A}: ~~ f \text{ is a norm continuous and bounded function} \}\\
\nonumber &&C_{0}(X, \mathcal{A})= \{f: X\rightarrow \mathcal{A}: ~~ f \text{ is continuous and vanishes at } \infty\},
\end{eqnarray}
 then the $\mathcal{M}(C_{0}(X,\mathcal{A}))\cong C_{b}(X, \mathcal{A})$.  Therefore $\mathcal{C}(C_{0}(X,\mathcal{A}))\cong C_{b}(X,\mathcal{A})/C_{0}(X,\mathcal{A})\cong \mathcal{C}(X^{*}, \mathcal{C}(\mathcal{A}))$. 
 
Our notion of trivial isomorphisms between corona algebras  follows the one in  \cite{Trivial}, where it  is considered more generally for isomorphisms between quotients of algebraic structures.  
\begin{definition}\label{trivial-definition}
Assume $\mathcal{A}$ and $\mathcal{B}$ are non-unital C*-algebras. A *-homomorphism $\Phi:\mathcal{C}(\mathcal{A})\rightarrow \mathcal{C}(\mathcal{B})$ is trivial if there is a *-homomorphism  $\Phi_{*}: \mathcal{M}(\mathcal{A})\rightarrow \mathcal{M}(\mathcal{B})$ such that
\begin{center}
\begin{tikzpicture}
\matrix(m)[matrix of math nodes,
row sep=2.6em, column sep=2.8em,
text height=1.5ex, text depth=0.25ex]
{\mathcal{M}(\mathcal{A})&\mathcal{M}(\mathcal{B})\\
\mathcal{C}(\mathcal{A})&\mathcal{C}(\mathcal{B})\\};
\path[->,font=\scriptsize,>=angle 90]
(m-1-1) edge node[auto] {$\Phi_{*}$} (m-1-2)
edge node[auto] {$\pi_{\mathcal A}$} (m-2-1)
(m-1-2) edge node[auto] {$\pi_{\mathcal B}$} (m-2-2)
(m-2-1) edge node[auto] {$\Phi$} (m-2-2);
\end{tikzpicture}
\end{center}
commutes, where $\pi_{\mathcal{A}}$ and $\pi_{\mathcal{B}}$ are the canonical quotient maps. 
\end{definition}
In the case of automorphisms, this notion of triviality is clearly weaker than the automorphisms being inner, i.e., being implemented by a unitary element. The question of whether the automorphisms, or generally the isomorphisms, between corona algebras are trivial turns out to be very sensitive to the additional set-theoretic assumptions. 
The most notorious results concerning the triviality of automorphisms of corona of non-commutative C*-algebras concern the Calkin algebra. It has been shown (\cite{PhilWeav} and \cite{FaCalkin}) that the statement that all the automorphisms of the Calkin algebra are inner, is independent from ZFC.  Refer to \cite[\S1]{Ghasemi} for a short overview of the some of the important results about triviality of the isomorphisms between corona algebras.

\subsection{Logic for metric structures} A main reference for logic for metric structures is \cite{Ben}, see also \cite{Bradd}.  A metric structure, in the sense of \cite{Ben}, is a many-sorted structure in which each sort is a
complete bounded metric space. A slightly modified version of the this logic was introduced in \cite{FHS} which does not require the structures to be bounded, and it is more suited while working with operator algebras. A key  feature of this logic is that each structure $\mathcal A$ is equipped with \emph{domains
of quantification}, bounded subsets of $\mathcal A$ on which all functions and predicates
are uniformly continuous (with a fixed modulus of uniform continuity) which the quantification is allowed.  For simplicity in notations
 we follow \cite{Ben} in this paper and assume that the metric structures are bounded, but the arguments can be adapted, in the obvious way, to unbounded metric structures as in \cite{FHS}. 
For a language $\mathcal{L}$ for metric structures, as usual a metric $\mathcal{L}$-structure is a complete bounded metric space (for each sort, in the case of many-sorted structures) with appropriate interpretations for predicates, functions and constant symbols in $\mathcal{L}$. Each predicate and function symbol is equipped with a \emph{modulus of uniform continuity}. In particular for every $\mathcal{L}$-formula $\varphi$, the set of all evaluations of $\varphi$ in any $\mathcal{L}$-structure with the diameter less than a fixed constant, is a bounded subset of the real numbers.
 Throughout this and the following section we assume formulas are $[0,1]$-valued. This is not always an assumption in continuous model theory  (e.g., formulas in the model theory for operator algebras). However, since the ranges of (evaluations of) formulas are always bounded in all interpretations  allowing $[0.\infty)$-valued or even negative evaluations results in equivalent logics; see \cite[4.1]{FaICM} for more details on this.
We assume that the reader is familiar with the basic definitions of model theory for metric structures, e.g., \emph{terms, atomic formulas, formulas, types}, etc. 

  For a topological space $X$,  let $\chi(X)$  denote the density character of $X$, i.e.,  the smallest cardinality of a dense subset of $X$. 
 Suppose $\kappa$ is an infinite cardinal. A model $\mathcal{A}$ is $\kappa$-\emph{saturated} if
every consistent type $\textbf{t}$ over $X\subseteq \mathcal{A}$ with  $|X| <\kappa$, is realized
in $\mathcal{A}$. If $\mathcal{A}$ is $\chi(X)$-saturated then $\mathcal{A}$ is called saturated.
We say $\mathcal A$ is countably saturated if it is $\aleph_1$-saturated.

We let $\mathbb{F}_{n}$ be the set of all formulas whose free variables are included
in $\{x_{1}, \dots,  x_{n}\}$.
For a metric structure $\mathcal{A}$ we usually abbreviate a tuple $(a_{1},\dots, a_{n})$ of elements of $\mathcal{A}$ by $\bar{a}$, when there is no confusion about the length of the tuple, and $\varphi(\bar{a})^{\mathcal{A}}$ is the \emph{interpretation} of $\varphi$ in $\mathcal{A}$ at $\bar{a}$.

Define the \emph{theory} of $\mathcal{A}$ to be
\begin{equation}
\nonumber Th(\mathcal{A})=\{\varphi\in \mathbb{F}_{0}:  ~ \varphi^{\mathcal{A}}=0 \}.
\end{equation}

Two metric structures $\mathcal{A}$ and $\mathcal{B}$ are \emph{elementarily equivalent}, $\mathcal{A}\equiv \mathcal{B}$, if $Th(\mathcal{A})= Th(\mathcal{B})$.
The \emph{universal theory} $Th_{\forall}(\mathcal{A})$ of $\mathcal{A}$ is the subset of $Th(\mathcal{A})$ consisting of sentences of the form $\sup_{\bar{x}}\varphi(\bar{x})$ where $\varphi$ is a quantifier-free formula and $\bar{x}$ is a finite tuple.

 Lets recall some requirements regarding the connectives from \cite[Chapter 6]{Ben}, which will be used throughout this paper. 
 Connectives are continuous functions from $[0,1]^{n}$ to $[0,1]$ for some $n\geq 1$.
 \begin{definition} A closed system of connectives is a family $\mathcal{F}=(F_{n}: n\geq 1)$ where each $F_{n}$ is a set of connectives $f:[0,1]^{n} \rightarrow [0,1]$ satisfying the following conditions.
\begin{enumerate}
\item[(i)] For each $n$, $F_{n}$ contains the projection onto the $j^{th}$ coordinate for each $j=1, \dots, n$.
\item[(ii)] For each $n$ and $m$, if $u\in F_{n}$, and $v_{1}, \dots, v_{n}\in F_{m}$, then the function $w: [0,1]^{m} \rightarrow [0,1]$ defined by $w(t)=u(v_{1}(t), \dots, v_{n}(t))$ belongs to $F_{m}$.
\end{enumerate}
\end{definition}

\begin{definition}
Given a closed system of connectives $\mathcal{F}$, the collection of $\mathcal{F}$-\emph{restricted} formulas is defined by induction.
\begin{enumerate}
\item Atomic formulas are $\mathcal{F}$-restricted.
\item If $u\in F_{n}$ and $\varphi_{1}, \dots, \varphi_{n}$ are $\mathcal{F}$-restricted formulas, then $u(\varphi_{1}, \dots, \varphi_{n})$ is also an  $\mathcal{F}$-restricted formula.
\item If $\varphi$ is an  $\mathcal{F}$-restricted formula, so are $\sup_{x} \varphi$ and $\inf_{x} \varphi$.
\end{enumerate}
\end{definition}

 Define a binary function $\dotminus: [0,1]^{2}\rightarrow [0,1]$ by
 \begin{equation}
  \nonumber        x \dotminus y =  \begin{cases}
                           x-y           & x\geq y\\
                           0             & \text{otherwise}
                           \end{cases}
 \end{equation}
 and let $\mathcal{F}_{0}=(F_{n}: n\geq 1)$ be the closed system of connectives generated from $\{0,1, x/2, \dotminus\}$ by closing it under (i) and (ii) (where $0$ and $1$ are constant functions with one variable).
\begin{proposition}\label{F-restricted}\cite[Proposition 6.6]{Ben}
  The set of all $\mathcal{F}_{0}$-restricted $\mathcal{L}$-formulas are uniformly dense in the set of all $\mathcal{L}$-formulas; that is, for any $\epsilon>0$ and any $\mathcal{L}$-formula $\varphi(x_{1}, \dots, x_{n})$, there is an  $\mathcal{F}_{0}$-restricted $\mathcal{L}$-formula $\psi(x_{1}, \dots, x_{n})$ such that for all $\mathcal{L}$-structures $\mathcal{A}$ we have
 \begin{equation}
 \nonumber |\varphi(a_{1},\dots, a_{n})- \psi(a_{1},\dots, a_{n})|< \epsilon
 \end{equation}
 for all $a_{1},\dots, a_{n}\in \mathcal{A}$.  In particular if $\mathcal{L}$ is countable, there is a countable set of $\mathcal{L}$-formulas which is uniformly dense in the set of all $\mathcal{L}$-formulas.
\end{proposition}

\subsection{Reduced products of metric structures}
Lets recall some definitions and basic theorems regarding reduced products of metric structures from \cite{FaShRig} and \cite{Lopes}.
Fix a language $\mathcal{L}$ in logic of metric structures. Throughout this paper $\mathcal{L}$ can be many-sorted, but in order to avoid distracting notations we shall assume it is one-sorted.
 Assume $\{(\mathcal{A}_{\gamma}, d_{\gamma}), ~\gamma\in \Omega\}$ is a family of metric $\mathcal{L}$-structures indexed by a set $\Omega$, all having diameter $\leq K$ for some constant $K$. Consider the direct product
\begin{eqnarray}
\nonumber \prod_{\Omega}\mathcal{A}_{\gamma}= \{\langle a(\gamma)\rangle: ~ ~a(\gamma)\in \mathcal{A}_{\gamma} \text{ for all } \gamma\in \Omega\}.
\end{eqnarray}
 Let   $\mathcal{I}$ be an ideal on $\Omega$. Define a map $d_{\mathcal{I}}$ on $\prod_{\Omega}\mathcal{A}_{\gamma}$ by
 \begin{equation}
 \nonumber d_{\mathcal{I}}(x,y)= \limsup_{\gamma\rightarrow \mathcal{I}} d_{\gamma}(x(\gamma), y(\gamma)) =\inf_{S\in \mathcal{I}} \sup_{\gamma \notin S} d_{\gamma}(x(\gamma), y(\gamma))
 \end{equation}
  where $x=\langle x(\gamma): \gamma\in \Omega\rangle$ and $y=\langle y(\gamma): \gamma\in \Omega\rangle$.  
Note that  the limit notation above corresponds to the notation $\lim_{\gamma\rightarrow \mathcal{I}} r_\gamma= L$ meaning  that for every $\epsilon>0$ the set 
$\{\gamma\in \Omega: |r_\gamma - L|\geq \epsilon \}$ belongs to the ideal $\mathcal I$.

The map $d_{\mathcal{I}}$ defines a pseudometric metric on  $\prod_{\Omega}\mathcal{A}_{\gamma}$.
  For $x,y\in \prod_{\Omega} \mathcal{A}_{\gamma}$ define $x\sim_{\mathcal{I}} y$ to mean $d_{\mathcal{I}}(x,y)=0$. Then $\sim_{\mathcal{I}}$ is an equivalence relation and the quotient
  \begin{equation}
  \nonumber \prod_{\mathcal{I}} \mathcal{A}_{\gamma}= (\prod_{\Omega} \mathcal{A}_{\gamma})/ \sim_{\mathcal{I}}
  \end{equation}
   with the induced metric $d_{\mathcal{I}}$  is a complete bounded metric space.
 We will use $\pi_{_{\mathcal{I}}}$ to denote the natural quotient map from $\prod_{\Omega} \mathcal{A}_{\gamma}$ onto $\prod_{\mathcal{I}} \mathcal{A}_{\gamma}$. For a tuple $\bar{a}=(a_{1},\dots,a_{k})$ of elements of $\prod_{\Omega}\mathcal{A}_{\gamma}$ we write $\pi_{_{\mathcal{I}}}(\bar{a})$ for $(\pi_{_{\mathcal{I}}}(a_{1}),\dots , \pi_{_{\mathcal{I}}}(a_{k}))$ and by $\bar{a}(\gamma)$ we denote the corresponding tuple $(a_{1}(\gamma), \dots, a_{k}(\gamma))$ of elements of $\mathcal{A}_{\gamma}$.

  Let $R$ be a predicate symbol in $\mathcal{L}$ and $\bar{a}$ be a tuple  of elements of $\prod_{\Omega}\mathcal{A}_{\gamma}$ of appropriate size. Define
  \begin{equation}
  \nonumber R(\pi_{_{\mathcal{I}}}(\bar{a}))= \limsup_{\mathcal{I}} R(\bar{a}(\gamma)).
  \end{equation}
If $f$ is a function symbol in $\mathcal{L}$ for an appropriate $\bar{a}$ define
\begin{equation}
\nonumber f(\pi_{_{\mathcal{I}}}(\bar{a}))= \pi_{_{\mathcal{I}}}(\langle f(\bar{a}(\gamma))\rangle),
\end{equation}
and if $c\in \mathcal{L}$ is a constant symbol let
\begin{equation}
\nonumber c^{\prod_{\mathcal{I}} \mathcal{A}_{\gamma}}= \pi_{_{\mathcal{I}}}(\langle c^{\mathcal{A}_{\gamma}}\rangle).
\end{equation}

  The quotient $\prod_{\mathcal{I}} \mathcal{A}_{\gamma}$ is called the \emph{reduced product} of the family $\{(\mathcal{A}_{\gamma}, d_{\gamma}):$ $ \gamma\in \Omega\}$ over the ideal $\mathcal{I}$.  Note that if $\mathcal{I}$ is a maximal (prime) ideal, then $\prod_{\mathcal{I}} \mathcal{A}_{\gamma}$ is the ultraproduct of the family $\{\mathcal{A}_{\gamma}, ~~\gamma\in \Omega\}$ over the ultrafilter $\mathcal{U}$ consisting of the complements of the elements of $\mathcal{I}$, usually denoted by $\prod_{\mathcal{U}} \mathcal{A}_{\gamma}$ or $(\prod_{\Omega} \mathcal{A}_{\gamma})_{\mathcal{U}}$ or $\prod_{\Omega} \mathcal{A}_{\gamma}/\mathcal{U}$.
 Also, in the case when $\mathcal{L}$ includes a distinguished constant symbol for $0$ (e.g., the language of C*-algebras) the reduced product of $\mathcal{L}$-structures  $\{\mathcal{A}_{\gamma}, ~~\gamma\in \Omega\}$ over $\mathcal{I}$ is the quotient of $\prod_{\Omega} \mathcal{A}_{\gamma}$ over its closed ideal $\bigoplus_{\mathcal{I}} \mathcal{A}_{\gamma}$ defined by
\begin{equation}
\nonumber \bigoplus_{\mathcal{I}}\mathcal{A}_{\gamma}= \{ a \in \prod_{\Omega}\mathcal{A}_{\gamma}: ~ ~  d_{\mathcal{I}}(a, 0^{\prod_{\Omega}\mathcal{A}_{\gamma}})=0\},
\end{equation}
and usually denoted by $\prod_{\Omega} \mathcal{A}_{\gamma}/\bigoplus_{\mathcal{I}} \mathcal{A}_{\gamma}$ (see \cite{FaShRig}).

\begin{proposition}
The structure $\langle \prod_{\mathcal{I}} \mathcal{A}_{\gamma}, d_{\mathcal{I}}\rangle$ in the language $\mathcal L$ with the interpretations of constants, functions and predicates as above, is a metric $\mathcal{L}$-structure.
\end{proposition}
\begin{proof}
We only have to check that each function and predicate symbol has the same modulus of uniform continuity. we shall prove this only for a
function symbol $f$ of arity $k$. Let $\Delta:[0,1]\rightarrow [0,1]$ be the modulus of uniform continuity of $f$, i.e.,  for  $\epsilon>0$ and $\bar{x}=(x_{1},\dots, x_{k}), ~~\bar{y}=(y_{1}, \dots, y_{k})$ tuples in each $\mathcal{A}_{\gamma}$ we have
\begin{equation}
\nonumber d_{\gamma}(\bar{x},\bar{y})< \Delta(\epsilon)~~ \rightarrow~~ d_{\gamma}(f(\bar{x}),f(\bar{y}))\leq\epsilon,
\end{equation}
where $d_{\gamma}(\bar{x},\bar{y})< \Delta(\epsilon)$ means $d_{\gamma}(x_{i}, y_{i})< \Delta(\epsilon)$ for ever $i=\{1,\dots, k\}$.

Suppose $\bar{a}$ and $\bar{b}$ in $(\prod_{\Omega} \mathcal{A}_{\gamma})^{k}$ are such that $d_{\mathcal{I}}(\pi_{I}(\bar{a}), \pi_{_{\mathcal{I}}}(\bar{b}))< \Delta(\epsilon)$. Then by the definition of $d_{\mathcal{I}}$ there is an $\mathcal{I}$-positive set $S\subseteq \Omega$ such that
 for every $\gamma\in S$ we have $d_{\gamma}(\bar{a}(\gamma), \bar{b}(\gamma))< \Delta(\epsilon)$, and therefore $d_{\gamma}(f(\bar{a}(\gamma)),f(\bar{b}(\gamma)))\leq\epsilon$. This implies that $d_{\mathcal{I}}(\pi_{_{\mathcal{I}}}(f(\bar{a})), \pi_{_{\mathcal{I}}}(f(\bar{y})))\leq\epsilon$.
\end{proof}
 \begin{lemma}\label{atomic}
  Assume $\mathcal{I}$ is an ideal on $\Omega$.
  If $ \varphi(\bar{y})$ is an atomic $\mathcal{L}$-formula and $\bar{a}$ is a tuple of elements of $\prod_{\Omega} \mathcal{A}_{\gamma}$, then
 \begin{equation}
 \nonumber \varphi(\pi_{_{\mathcal{I}}}(\bar{a}))^{\prod_{\mathcal{I}} \mathcal{A}_{\gamma}}= \limsup_{\mathcal{I}} \varphi(\bar{a}(\gamma))^{\mathcal{A}_{\gamma}}.
   \end{equation}
  \end{lemma}

  \begin{proof}
  This easily follows from the definition of $d_{\mathcal{I}}$ and the interpretation of atomic formulas.
  \end{proof}

\section{An extension of Feferman-Vaught theorem for reduced products of metric structures}\label{FV}

The evaluation of a non-atomic formula in reduced products turns out to be more complicated than the atomic case, see \cite{Lopes}.  In this section we give an extension of Feferman-Vaught theorem to reduced products of metric structures, which just like its classical version, gives a powerful tool to prove elementary equivalence of reduced products.

Suppose $\{\mathcal{A}_{\gamma}: \gamma \in \Omega\}$ is a family of metric structures in a fixed language $\mathcal{L}$ and $\mathcal{I}$ is an ideal on $\Omega$. For the purposes of this section let
\begin{equation}
\nonumber \mathcal{A}_{\Omega}= \prod_{\Omega} \mathcal{A}_{\gamma}, \qquad\qquad \mathcal{A}_{\mathcal{I}}= \prod_{\mathcal{I}} \mathcal{A}_{\gamma}.
\end{equation}
For an $\mathcal{L}$-formula $\varphi(\bar{x})$, a tuple $\bar{a}$ of elements of $\mathcal{A}$  and 
\begin{equation}
\nonumber X=\{\gamma\in \Omega : \varphi(\bar{a}(\gamma))^{\mathcal{A}_{\gamma}}> r\}
\end{equation}
for some $r\in\mathbb{R}$, we use $\tilde{X}$ to denote the set 
\begin{equation}
\nonumber \tilde{X}=\{\gamma\in \Omega : \varphi(\bar{a}(\gamma))^{\mathcal{A}_{\gamma}}\geq r\}.
\end{equation}
 \begin{definition}\label{det}
  For an $\mathcal{F}_{0}$-restricted $\mathcal{L}$-formula $\varphi(x_{1},\dots, x_{l})$, we say $\varphi$ is determined up to $2^{-n}$ by $( \sigma_{0},\dots,\sigma_{2^{n}}; \psi_{0}, \dots, \psi_{m-1})$ if
\begin{enumerate}
\item Each $\sigma_{i}$ is a formula in the language of Boolean algebras with at most $s=m2^{n}$ many variables, which is monotonic, i.e.,
             \begin{eqnarray}
             \nonumber && T_{BA}\vdash \forall y_{1}\dots, y_{s}, z_{1},\dots,z_{s}(\sigma_{i}(y_{1},\dots, y_{s})\wedge\bigwedge_{i=1}^{s}y_{i}\leq z_{i}\\
            \nonumber  && \qquad\qquad\qquad\qquad\qquad\qquad\rightarrow   \sigma_{i}(z_{1},\dots, z_{s})).
             \end{eqnarray}
             (Here, $T_{BA}$ denotes the theory of Boolean algebras.)
\item  Each $\psi_{j}(x_{1},\dots, x_{l})$ is an $\mathcal{F}_{0}$-restricted $\mathcal{L}$-formula for $j= 0, \dots, m-1$.
\item For any indexed set $\Omega$, an ideal $\mathcal{I}$ on $\Omega$, a family $\{\mathcal{A}_{\gamma}: \gamma\in \Omega\}$ of metric $\mathcal{L}$-structures and $a_{1}, \dots, a_{l}\in \mathcal{A}_{\Omega}$ the following hold: \\
    for every $ \ell=0, \dots, 2^{n}$
    \begin{eqnarray}
     \qquad\qquad P(\Omega)/\mathcal{I}\vDash \sigma_{\ell}([X^{0}_{0}]_{\mathcal{I}},\dots, [X^{0}_{2^{n}}]_{\mathcal{I}},\dots, [X^{m-1}_{0}]_{\mathcal{I}},\dots, [X^{m-1}_{2^{n}}]_{\mathcal{I}})&& \nonumber\\
     \qquad\qquad\Longrightarrow  &\varphi(\pi_{_{\mathcal{I}}}(\bar{a}))^{\mathcal{A}_{\mathcal{I}}}> \ell/2^{n},& \nonumber
     \end{eqnarray}
     and 
     \begin{eqnarray}         
   &\varphi(\pi_{_{\mathcal{I}}}(\bar{a}))^{\mathcal{A}_{\mathcal{I}}}> \ell/2^{n}&  \nonumber\\
  &\qquad\qquad \qquad\qquad\Longrightarrow 
   P(\Omega)/\mathcal{I}\vDash &  \sigma_{\ell}([\tilde{X}^{0}_{0}]_{\mathcal{I}},\dots, [\tilde{X}^{0}_{2^{n}}]_{\mathcal{I}},\dots, [\tilde{X}^{m-1}_{0}]_{\mathcal{I}},\dots, [\tilde{X}^{m-1}_{2^{n}}]_{\mathcal{I}})   \nonumber
    \end{eqnarray}
    
    where  $ X^{j}_{i}=\{\gamma\in \Omega : \psi_{j}(\bar{a}(\gamma))^{\mathcal{A}_{\gamma}}> i/2^{n}\}$ for each $j=0,\dots, m-1$ and $i=0, \dots, 2^{n}$.

\end{enumerate}
\end{definition}

Using Lemma \ref{F-restricted} we can generalize this definition to all $\mathcal{L}$-formulas.
\begin{definition}\label{generaldet}
We say an $\mathcal{L}$-formula $\varphi$ is determined up to $2^{-n}$ if there is an $\mathcal{F}_{0}$-restricted $\mathcal{L}$-formula $\tilde{\varphi}$ which is uniformly within $2^{-n-1}$ of $\varphi$ and $\tilde{\varphi}$ is determined up to $2^{-n}$ by some $( \sigma_{0},\dots,\sigma_{2^{n}}; \psi_{0}, \dots, \psi_{m-1})$.
\end{definition}

\begin{theorem}\label{FV}
Every formula is determined up to $2^{-n}$ for any given $n\in\mathbb{N}$.
\end{theorem}

\begin{proof}
By Definition \ref{generaldet} and Lemma \ref{F-restricted}, without loss of generality, we can assume that formulas are $\mathcal{F}_{0}$-restricted.
Assume $\varphi$ is an atomic $\mathcal{L}$-formula and for each $i\leq 2^{n}$ define
\begin{equation}
\nonumber \sigma_{i}(y_{0},\dots, y_{2^{n}}):=  y_{i}\neq 0.
\end{equation}
We show that $\varphi$ is determined up to $2^{-n}$ by $(\sigma_{0}, \dots, \sigma_{2^{n}}; \varphi)$.
Conditions (1) and (2) of Definition \ref{det} are clearly satisfied.
  For an indexed set $\Omega$, an ideal $\mathcal{I}$ on $\Omega$, a family $\{\mathcal{A}_{\gamma}: \gamma\in \Omega\}$ of metric $\mathcal{L}$-structures and $a_{1}, \dots, a_{l}\in \mathcal{A}_{\Omega}$ let
\begin{equation}
\nonumber X_{i}=\{\gamma\in \Omega : \varphi(\bar{a}(\gamma))^{\mathcal{A}_{\gamma}}> i/2^{n}\},
\end{equation}
since $\varphi(\pi_{_{\mathcal{I}}}(\bar{a}))^{\mathcal{A}_{\mathcal{I}}}=\limsup_{\mathcal{I}} \varphi(\bar{a}(\gamma))^{\mathcal{A}_{\gamma}}$ we have
\begin{eqnarray}
\nonumber  P(\Omega)/\mathcal{I}\vDash \sigma_{\ell}([X_{0}]_{\mathcal{I}},\dots, [X_{2^{n}}]_{\mathcal{I}})  &\Longleftrightarrow&  X_{\ell}\notin\mathcal{I}\\
\nonumber &\Longleftrightarrow& \varphi(\pi_{_{\mathcal{I}}}(\bar{a}))^{\mathcal{A}_{\mathcal{I}}}> \ell/2^{n}.
\end{eqnarray}
Since each $X_{i}^{j}\subseteq \tilde{X}_{i}^{j}$, by the monotonicity of $\sigma_{\ell}$, $\varphi(\pi_{_{\mathcal{I}}}(\bar{a}))^{\mathcal{A}_{\mathcal{I}}}> \ell/2^{n}$ also implies that $ P(\Omega)/\mathcal{I}\vDash \sigma_{\ell}([\tilde{X}_{0}]_{\mathcal{I}},\dots, [\tilde{X}_{2^{n}}]_{\mathcal{I}})$.  
 Thus $\varphi$ is determined up to $2^{-n}$ by $(\sigma_{0}, \dots, \sigma_{2^{n}}; \varphi)$.

Assume $\varphi(\bar{x})=f(\alpha(\bar{x}))$ where $f\in\{0,1, x/2\}$ and $\alpha$ is some $\mathcal{L}$-formula determined up to $2^{-n}$ by $(\sigma_{0}, \dots, \sigma_{2^{n}}; \psi_{0}, \dots, \psi_{m-1})$. The cases where $f\in\{0,1\}$ are trivial; for example if $f=0$ then $\varphi(\bar{x})$ is determined up to $2^{-n}$ by $(\tau_{0}, \dots, \tau_{2^{n}}; 0)$ where   $\tau_{i}:= 1\neq 1$ for each $i= 0,\dots, 2^{n}$.
  If $f(x)=x/2$  then  it is also straightforward to check that $\varphi$ is determined up to $2^{-n-1}$ by $(\sigma_{0}, \dots, \sigma_{2^{n}},\tau_{2^{n}+1}, \dots, \tau_{2^{n+1}}; f(\psi_{0}), \dots, f(\psi_{m-1}))$, where each $\tau_{i}$ is a false sentence (e.g, $1\neq 1$).

Let $\varphi(\bar{x})= \alpha_{1}(\bar{x}) \dotminus \alpha_{2}(\bar{x})$ where each $\alpha_{t}$ $(t\in\{1,2\})$ is determined up to $2^{-n}$ by $(\sigma_{0}^{t}, \dots, \sigma_{2^{n}}^{t}; \psi^{t}_{0}, \dots, \psi^{t}_{m_{t}-1})$. We claim that $\varphi$ is determined up to $2^{-n}$ by $(\tau_{0}, \dots, \tau_{2^{n}}; \psi^{1}_{0}, \dots \\, \psi^{1}_{m_{1}-1}, 1-\psi^{2}_{0}, \dots, 1-\psi^{2}_{m_{2}-1})$ where the Boolean algebra formulas $\tau_{k}$ are defined by
\begin{eqnarray}
\nonumber &&\tau_{k}(x_{0}^{0}, \dots, x_{2^{n}}^{0}, \dots,  x_{0}^{m_{1}-1},\dots ,x_{2^{n}}^{m_{1}-1} , z_{0}^{0}, \dots, z_{2^{n}}^{0},\dots, z_{0}^{m_{2}-1},\dots ,z_{2^{n}}^{m_{2}-1}):=\\
 \nonumber && \qquad\qquad\qquad\qquad\bigvee_{i_{0}=k}^{2^{n}}\big[\sigma_{i_{0}}^{1}(x_{0}^{0}, \dots, x_{2^{n}}^{0}, \dots,  x_{0}^{m_{1}-1},\dots ,x_{2^{n}}^{m_{1}-1})\\
 \nonumber && \qquad\qquad\qquad\qquad \wedge \neg\sigma_{i_{0}-k}^{2}(-z_{2^{n}}^{0}, \dots, -z_{0}^{0}, \dots, -z_{2^{n}}^{m_{2}-1}, \dots, -z_{0}^{m_{2}-1} )\big],
\end{eqnarray}
(here $-z$ is the Boolean algebra complement of $z$).
 Conditions (1) and (2) in Definition \ref{det} are clearly satisfied. For (3) let  $\mathcal{A}_{\mathcal{I}}$  be a reduced product of $\mathcal{L}$-structures (indexed by $\Omega$ and over an ideal $\mathcal{I}$) and $a_{1}, \dots, a_{l}\in \mathcal{A}_{\Omega}$. Let
 \begin{equation}
 \nonumber X_{i}^{j}= \{\gamma\in \Omega : \psi^{1}_{j}(\bar{a}(\gamma))^{\mathcal{A}_{\gamma}} > i/2^{n}\} \qquad\qquad 0\leq j\leq m_{1}-1,
 \end{equation}
  \begin{equation}
 \nonumber Y_{i}^{j}= \{\gamma\in \Omega : \psi_{j}^{2}(\bar{a}(\gamma))^{\mathcal{A}_{\gamma}} > i/2^{n} \} \qquad\qquad 0\leq j\leq m_{2}-1,
 \end{equation}
 and
   \begin{equation}
 \nonumber \quad Z_{i}^{j}= \{\gamma\in \Omega : 1-\psi_{j}^{2}(\bar{a}(\gamma))^{\mathcal{A}_{\gamma}} > i/2^{n}\} \qquad\qquad 0\leq j\leq m_{2}-1.
 \end{equation}
Note that $\tilde{Y}_{2^{n}- i}^{j}= (Z_{i}^{j})^{c}$ for each $i$ and $j$. Assume 
\begin{equation}
  \nonumber  P(\Omega)/\mathcal{I}\vDash \tau_{k}([X^{0}_{0}]_{\mathcal{I}},\dots, [X^{m_{1}-1}_{2^{n}}]_{\mathcal{I}}, [Z^{0}_{0}]_{\mathcal{I}},\dots, [Z^{m_{2}-1}_{2^{n}}]_{\mathcal{I}}),
\end{equation}
then  for some  $i_{0}\geq k$,
\begin{equation}
  \nonumber \qquad\qquad P(\Omega)/\mathcal{I}\vDash \sigma^{1}_{i_{0}}([X^{0}_{0}]_{\mathcal{I}},\dots, [X^{m_{1}-1}_{2^{n}}]_{\mathcal{I}})  
     \wedge  \neg\sigma^{2}_{i_{0}-k}([(Z^{0}_{2^{n}})^{c}]_{\mathcal{I}},\dots, [(Z^{m_{1}-1}_{0})^{c}]_{\mathcal{I}}),
\end{equation}
\begin{eqnarray}   
  \qquad\qquad\quad \Longrightarrow  P(\Omega)/\mathcal{I}\vDash \sigma^{1}_{i_{0}}([X^{0}_{0}]_{\mathcal{I}},\dots, [X^{m_{1}-1}_{2^{n}}]_{\mathcal{I}})
 \wedge \neg \sigma^{2}_{i_{0}-k}([\tilde{Y}^{0}_{0}]_{\mathcal{I}},\dots, [\tilde{Y}^{m_{1}-1}_{2^{n}}]_{\mathcal{I}}), \nonumber
 \end{eqnarray}
and therefore
\begin{eqnarray}    
    \alpha_{1}(\pi_{_{\mathcal{I}}}(\bar{a}))^{\mathcal{A}_{\mathcal{I}}}> i_{0}/2^{n} \text{ and } \alpha_{2}(\pi_{_{\mathcal{I}}}(\bar{a}))^{\mathcal{A}_{\mathcal{I}}}\leq (i_{0}-k)/2^{n}.\nonumber 
\end{eqnarray}
Hence $\varphi(\pi_{_{\mathcal{I}}}(\bar{a}))^{\mathcal{A}_{\mathcal{I}}}> k/2^{n}$. To prove the other direction assume $\varphi(\pi_{_{\mathcal{I}}}(\bar{a}))^{\mathcal{A}_{\mathcal{I}}}> k/2^{n}$. For some  $i_{0}\geq k$,
\begin{eqnarray}    
    \alpha_{1}(\pi_{_{\mathcal{I}}}(\bar{a}))^{\mathcal{A}_{\mathcal{I}}}> i_{0}/2^{n} \text{ and } \alpha_{2}(\pi_{_{\mathcal{I}}}(\bar{a}))^{\mathcal{A}_{\mathcal{I}}}\leq (i_{0}-k)/2^{n}.\nonumber 
\end{eqnarray}
By the induction assumptions 
\begin{eqnarray}   
P(\Omega)/\mathcal{I}\vDash \sigma^{1}_{i_{0}}([\tilde{X}^{0}_{0}]_{\mathcal{I}},\dots, [\tilde{X}^{m_{1}-1}_{2^{n}}]_{\mathcal{I}})
 \wedge \neg \sigma^{2}_{i_{0}-k}([\tilde{Y}^{0}_{0}]_{\mathcal{I}},\dots, [\tilde{Y}^{m_{1}-1}_{2^{n}}]_{\mathcal{I}}), \nonumber
 \end{eqnarray}
and note that  $(\tilde{Z}_{2^{n}-i}^{j})^{c}\subseteq \tilde{Y}_{i}^{j}$ for each $i$ and $j$, which implies 
\begin{equation}
  P(\Omega)/\mathcal{I}\vDash \sigma^{1}_{i_{0}}([\tilde{X}^{0}_{0}]_{\mathcal{I}},\dots, [\tilde{X}^{m_{1}-1}_{2^{n}}]_{\mathcal{I}})  
     \wedge  \neg\sigma^{2}_{i_{0}-k}([(\tilde{Z}^{0}_{2^{n}})^{c}]_{\mathcal{I}},\dots, [(\tilde{Z}^{m_{1}-1}_{0})^{c}]_{\mathcal{I}}), \nonumber
\end{equation}
so
\begin{equation}
  \nonumber   P(\Omega)/\mathcal{I}\vDash \tau_{k}([\tilde{X}^{0}_{0}]_{\mathcal{I}},\dots, [\tilde{X}^{m_{1}-1}_{2^{n}}]_{\mathcal{I}}, [\tilde{Z}^{0}_{0}]_{\mathcal{I}},\dots, [\tilde{Z}^{m_{2}-1}_{2^{n}}]_{\mathcal{I}}).
\end{equation}
Therefore $\varphi$ is determined up to $2^{-n}$ by $(\tau_{0}, \dots, \tau_{2^{n}}; \psi^{1}_{0}, \dots, \psi^{1}_{m_{1}-1},1-\psi^{2}_{0}, \dots, 1-\psi^{2}_{m_{2}-1})$.

Assume $\varphi(\bar{x})=\sup_{z}\psi(\bar{x}, z)$ where $\psi$ is determined  by $(\sigma_{0},\dots,\sigma_{2^{n}}; \psi_{0}, 
\dots, \psi_{m-1})$ up to $2^{-n}$.  Let $d=2^{n+m}-1$ and $s_{0},\dots, s_{d-1}$ be an enumeration of non-empty elements of $\prod_{i=0}^{2^{n}} P( \{0, \dots, m-1\})$, i.e,  each $s_{k}=(s_{k}(0),\dots, s_{k}(2^{n}))$ where $s_{k}(i)\subseteq \{0, \dots, m-1\}$ for each $i$. Also assume that for each $0\leq k\leq m-1$ we have $s_{k}=\{\{k\}, \emptyset, \dots, \emptyset\}$.  For any $k\in\{0, \dots, d-1\}$ define an  $\mathcal{L}$-formula $\theta_{k}$ by
\begin{equation}
\nonumber \theta_{k}(\bar{x})= \sup_{z} \min \big\{\psi_{j}(\bar{x}, z): j\in \bigcup_{i=0}^{2^{n}} s_{k}(i)\big\}.
\end{equation}
Note that if $0\leq k\leq m-1$ then $\theta_{k}(\bar{x})=  \sup_{z}\psi_{k}(\bar{x}, z)$.
For each $i\in \{0, \dots, 2^{n}\}$ define a Boolean algebra formula $\tau_{i}$ by
\begin{eqnarray}
   \nonumber \tau_{i}(y_{0}^{0},\dots, y^{0}_{2^{n}}, \dots, y_{0}^{d-1},\dots, y^{d-1}_{2^{n}})&=& \exists  z_{0}^{0},\dots, z^{0}_{2^{n}}, \dots, z_{0}^{d-1},\dots, z^{d-1}_{2^{n}}\\
   \nonumber && [\bigwedge_{j=0}^{d-1}\bigwedge_{i=0}^{2^{n}}(z_{i}^{j}\leq y_{i}^{j})\wedge \bigwedge_{i=0}^{2^{n}}\bigwedge_{\substack{ s_{k}(t)\cup s_{k^{\prime}}(t)=s_{k^{\prime\prime}}(t)\\ \forall t}} (z_{i}^{k}.z_{i}^{k^{\prime}}= z_{i}^{k^{\prime\prime}}) \\
  \nonumber && \wedge~~ \sigma_{i}(z^{0}_{0},\dots, z^{0}_{2^{n}},\dots, z^{m-1}_{0},\dots, z^{m-1}_{2^{n}})].
\end{eqnarray}

We claim that $\varphi$ is determined up to $2^{-n}$ by $(\tau_{0},\dots,\tau_{2^{n}}; \theta_{0}, \dots, \theta_{d-1})$.
Again condition (1) is clearly satisfied. Condition (2) is also satisfied, since $\min\{x,y\}= x \dotminus(x \dotminus y)$. For (3) assume a reduced product $\mathcal{A}_{\mathcal{I}}$ and $a_{1}, \dots, a_{l}\in \mathcal{A}_{\Omega}$ are given.

 First assume  $ \varphi(\pi_{_{\mathcal{I}}}(\bar{a}))^{\mathcal{A}_{\mathcal{I}}}> \ell/2^{n}$ for some $\ell$.
  Let
  \begin{equation}
  \nonumber \delta= \frac{\min\{\varphi(\pi_{_{\mathcal{I}}}(\bar{a}))^{\mathcal{A}_{\mathcal{I}}}- \ell/2^{n}, 1/2^{n}\}}{2}
  \end{equation}
  and find $c= (c(\gamma))_{\gamma\in\Omega}$ such that
\begin{equation}
\nonumber \psi(\pi_{_{\mathcal{I}}}(\bar{a}, c))^{\mathcal{A}_{\mathcal{I}}}>  \varphi(\pi_{_{\mathcal{I}}}(\bar{a}))^{\mathcal{A}_{\mathcal{I}}}- \delta > \ell/2^{n}.
\end{equation}
 For each $i\leq 2^{n}$ and $k\leq d-1$ let
\begin{equation}
\nonumber Y_{i}^{k}= \{\gamma\in \Omega: \theta_{k}(\bar{a}(\gamma))^{\mathcal{A}_{\gamma}}> i/2^{n}\},
\end{equation}
and let
\begin{equation}
\nonumber Z_{i}^{k}= \{\gamma\in \Omega:  \min \{\psi_{j}(\bar{x}(\gamma), c(\gamma)): j\in \bigcup_{t=0}^{2^{n}} s_{k}(t)\}^{\mathcal{A}_{\gamma}}> i/2^{n}\}.
\end{equation}
From definition of $\theta_{k}$ it is clear that $\tilde{Z}_{i}^{k}\subseteq \tilde{Y}_{i}^{k}$, and
\begin{equation}
\nonumber \bigwedge_{i=0}^{2^{n}}\bigwedge_{\substack{ s_{k}(t)\cup s_{k^{\prime}}(t)=s_{k^{\prime\prime}}(t)\\ \forall t}} (\tilde{Z}_{i}^{k}\cap \tilde{Z}_{i}^{k^{\prime}}= \tilde{Z}_{i}^{k^{\prime\prime}}),
\end{equation}
and by the inductive assumption
\begin{equation}
\nonumber P(\Omega)/\mathcal{I}\vDash \sigma_{\ell}([\tilde{Z}^{0}_{0}]_{\mathcal{I}},\dots, [\tilde{Z}^{0}_{2^{n}}]_{\mathcal{I}},\dots, [\tilde{Z}^{m-1}_{0}]_{\mathcal{I}},\dots, [\tilde{Z}^{m-1}_{2^{n}}]_{\mathcal{I}}).
\end{equation}
Hence $P(\Omega)/\mathcal{I}\vDash \tau_{\ell}([\tilde{Y}^{0}_{0}]_{\mathcal{I}},\dots, [\tilde{Y}^{0}_{2^{n}}]_{\mathcal{I}},\dots, [\tilde{Y}^{d-1}_{0}]_{\mathcal{I}},\dots, [\tilde{Y}^{d-1}_{2^{n}}]_{\mathcal{I}})$.

For the other direction let
\begin{equation}
\nonumber Y_{i}^{k}= \{\gamma\in \Omega: \theta_{k}(\bar{a}(\gamma))^{\mathcal{A}_{\gamma}}> i/2^{n}\},
\end{equation}
 and suppose $P(\Omega)/\mathcal{I}\vDash \tau_{\ell}([Y^{0}_{0}]_{\mathcal{I}},\dots, [Y^{0}_{2^{n}}]_{\mathcal{I}},\dots, [Y^{d-1}_{0}]_{\mathcal{I}},\dots, [Y^{d-1}_{2^{n}}]_{\mathcal{I}})$. There are sets $Z_{0}^{0},\dots, Z^{0}_{2^{n}}, \dots, Z_{0}^{d-1},\dots, Z^{d-1}_{2^{n}}$ such that the following hold.
 \begin{eqnarray}\label{1}
  \nonumber &&[Z_{i}^{k}]_{\mathcal{I}} \subseteq [Y_{i}^{k}]_{\mathcal{I}} \qquad\qquad\qquad\qquad\qquad\qquad 0\leq i\leq 2^{n},  0\leq k \leq d-1,\\
  \nonumber &&[Z_{i}^{k}]_{\mathcal{I}} \cap  [Z_{i}^{k^{\prime}}]_{\mathcal{I}}=   [Z_{i}^{k^{\prime\prime}}]_{\mathcal{I}} \qquad\qquad\qquad\qquad \forall t~~s_{k}(t)\cup  s_{k^{\prime}}(t)=s_{k^{\prime\prime}}(t), \\
  \nonumber &&P(\Omega)/\mathcal{I}\vDash \sigma_{\ell}([Z^{0}_{0}]_{\mathcal{I}},\dots, [Z^{0}_{2^{n}}]_{\mathcal{I}},\dots, [Z^{m-1}_{0}]_{\mathcal{I}},\dots, [Z^{m-1}_{2^{n}}]_{\mathcal{I}}).
 \end{eqnarray}
Since there are only finitely many conditions above, we can find a set $S\in \mathcal{I}$ such that if  $D= \Omega\setminus S$ then

 \begin{eqnarray}\label{2}
   &&Z_{i}^{k} \cap D \subseteq Y_{i}^{k} \qquad\qquad\qquad\qquad\qquad 0\leq i\leq 2^{n},  0\leq k \leq d-1,\\
  \nonumber &&Z_{i}^{k}\cap  Z_{i}^{k^{\prime}}\cap D =   Z_{i}^{k^{\prime\prime}} \cap D \qquad\qquad\qquad \forall t~~ s_{k}(t)\cup  s_{k^{\prime}}(t)=s_{k^{\prime\prime}}(t),
 \end{eqnarray}
 Fix $\gamma \in D$, and for each $i\in \{0,\dots , 2^{n}\}$ let $u(i)= \{j\in \{0,\dots, m-1\} : \gamma \in Z^{j}_{i}\}$.
 If $k\in\{0, \dots , d-1\}$ be such that $s_{k}=(u(0), \dots, u(2^{n}))$, then since $\gamma\in Z^{j}_{i}$  for all $j\in u(i)$, using
 (\ref{2})  we have $\gamma \in Y_{i}^{k}$ (for all $i$) and hence

 \begin{equation}
 \nonumber \theta_{k}(\bar{a}(\gamma))^{\mathcal{A}_{\mathcal{I}}}=\sup_{z}\min_{j\in \cup_{t=0}^{2^{n}} u(t)} \psi_{j}(\bar{a}(\gamma), z)> i/2^{n}.
 \end{equation}
  Let
  \begin{equation}
  \nonumber \delta= \frac{\min_{i,k}\{\theta_{k}(\bar{a}(\gamma))^{\mathcal{A}_{\mathcal{I}}}- i/2^{n}, 1/2^{n}\}}{2}.
  \end{equation}
   We can pick $c(\gamma)\in A_{\gamma}$  such that for every $i$

 \begin{equation}
 \label{3} \min_{j\in u(i)} \psi_{j}(\bar{a}(\gamma), c(\gamma))> \theta_{k}(a(\gamma)) - \delta \geq i/2^{n}.
 \end{equation}
For $\gamma\notin D$ define $c(\gamma)$ arbitrarily and let $c=(c(\gamma))_{\gamma \in \Omega }$.
 For each $j\in \{0,\dots, m-1\}$ and $i\in\{0,\dots, 2^{n}\}$ let

 \begin{equation}
 \nonumber X^{j}_{i}=\{\gamma \in \Omega : \psi_{j}(\bar{a}(\gamma), c(\gamma))^{\mathcal{A}_{\gamma}}> i/2^{n}\}.
 \end{equation}
 Now (\ref{2}) and (\ref{3}) imply that $Z_{i}^{j}\cap D \subseteq X_{i}^{j}$ for all $i$ and $j$. Therefore
 \begin{equation}
 \nonumber P(\Omega)/\mathcal{I} \vDash \bigwedge_{j=0}^{m-1}\bigwedge_{i=0}^{2^{n}} [Z_{i}^{j}]_{\mathcal{I}} \leq [X_{i}^{j}]_{\mathcal{I}}.
 \end{equation}
 Since $P(\Omega)/\mathcal{I}\vDash \sigma_{\ell}([Z^{0}_{0}]_{\mathcal{I}},\dots, [Z^{0}_{2^{n}}]_{\mathcal{I}},\dots, [Z^{m-1}_{0}]_{\mathcal{I}},\dots, [Z^{m-1}_{2^{n}}]_{\mathcal{I}})$, by monotonicity of $\sigma_{\ell}$ we have
 \begin{equation}
 \nonumber P(\Omega)/\mathcal{I} \vDash \sigma_{\ell}([X^{0}_{0}]_{\mathcal{I}},\dots, [X^{0}_{2^{n}}]_{\mathcal{I}},\dots, [X^{m-1}_{0}]_{\mathcal{I}}, \dots,   [X^{m-1}_{2^{n}}]_{\mathcal{I}}).
 \end{equation}
 Therefore by the induction assumption we have
 \begin{equation}
 \nonumber \psi(\pi_{_{\mathcal{I}}}(\bar{a},c))^{\mathcal{A}_{\mathcal{I}}}> \ell/2^{n},
 \end{equation}
 which implies that $ \varphi(\pi_{_{\mathcal{I}}}(\bar{a}))^{\mathcal{A}_{\mathcal{I}}}> \ell/2^{n}$.
\end{proof}

\begin{remark}\label{remark-FV}  From the proof of Theorem \ref{FV} at each step of the induction, it is straightforward (however lengthy) to check  that  for each   $\ell \in \{1,\dots, 2^n\}$, the Boolean algebra formulas $\sigma_\ell$ have the property that 
\begin{align}
\nonumber T_{BA} \vDash \sigma_{\ell-1}(z_0^0,\dots,z_{2^n}^{m-1})\leftrightarrow 
\sigma_\ell (&\Omega,
z_0^0,z_1^0,\dots,z_{2^n-1}^0,  \Omega,
z_0^1,z_1^1,\dots,z_{2^n-1}^1,\\
\nonumber & \dots, \Omega,
z_0^{m-1},z_1^{m-1},\dots,z_{2^n-1}^{m-1}). 
\end{align}
\end{remark}
\begin{lemma}\label{weaker lemma}
Assume $\varphi(\bar x)$ is an $\mathcal F_0$-restricted $\mathcal L$-formula which is determined up to $2^n$ by $( \sigma_{0},\dots,\sigma_{2^{n}}; \psi_{0}, \dots, \psi_{m-1})$.  Assume $\mathcal A_\mathcal I$ is a reduced product over an ideal $\mathcal I$  and $a_{1}, \dots, a_{l}\in \mathcal{A}_{\Omega}$. 
Then for each $\ell\in\{1, \dots, 2^n\}$
    \begin{align}
  \nonumber  P(\Omega)/\mathcal{I}\vDash &\sigma_{\ell}([\tilde X^{0}_{0}]_{\mathcal{I}},\dots, [\tilde X^{0}_{2^{n}}]_{\mathcal{I}},\dots, [\tilde X^{m-1}_{0}]_{\mathcal{I}},\dots, [\tilde X^{m-1}_{2^{n}}]_{\mathcal{I}}) \\
   \nonumber  &\Longrightarrow  \varphi(\pi_{_{\mathcal{I}}}(\bar{a}))^{\mathcal{A}_{\mathcal{I}}}> (\ell-1)/2^{n}, 
     \end{align}
where     $ X^{j}_{i}=\{\gamma\in \Omega : \psi_{j}(\bar{a}(\gamma))^{\mathcal{A}_{\gamma}}> i/2^{n}\}$, $j\in\{0,\dots, m-1\}$ and $i\in \{0, \dots, 2^{n}\}$.
\end{lemma}
\begin{proof}
Assume 
$$
 P(\Omega)/\mathcal{I}\vDash \sigma_{\ell}([\tilde X^{0}_{0}]_{\mathcal{I}},\dots, [\tilde X^{0}_{2^{n}}]_{\mathcal{I}},\dots, [\tilde X^{m-1}_{0}]_{\mathcal{I}},\dots, [\tilde X^{m-1}_{2^{n}}]_{\mathcal{I}}) 
$$
and note that $\tilde X_0^j = \Omega$ and 
 $\tilde X_i^j \subseteq X_{i-1}^j$ for each $j\in\{0,\dots, m-1\}$ and $i\in \{1, \dots, 2^{n}\}$. 
 Thus by the  monotonicity of $\sigma_\ell$ we have
   \begin{equation}
 \nonumber P(\Omega)/\mathcal{I} \vDash \sigma_{\ell}([\Omega]_\mathcal I, [{X}^{0}_{0}]_{\mathcal{I}},\dots,    [{X}^{0}_{2^{n}-1}]_{\mathcal{I}}, \dots,  [\Omega]_\mathcal I, [{X}^{m-1}_{0}]_{\mathcal{I}}, \dots, [{X}^{m-1}_{2^{n}}]_{\mathcal{I}}),
  \end{equation}
    which   implies that (Remark \ref{remark-FV})
    \begin{equation}
 \nonumber P(\Omega)/\mathcal{I} \vDash \sigma_{\ell-1}([{X}^{0}_{0}]_{\mathcal{I}},\dots, [{X}^{m-1}_{2^{n}}]_{\mathcal{I}}).
  \end{equation}
 Therefore $ \varphi(\pi_{_{\mathcal{I}}}(\bar{a}))^{\mathcal{A}_{\mathcal{I}}}> (\ell-1)/2^{n}$.

\end{proof}

Let us give some easy applications of Theorem \ref{FV}.
Assume $\{\mathcal{A}_{\gamma}: \gamma\in \Omega\}$ and $\{\mathcal{B}_{\gamma}: \gamma\in \Omega\}$ are families of metric $\mathcal{L}$-structures indexed by $\Omega$ and for an ideal $\mathcal{I}$ on $\Omega$ let $\mathcal{A}_{\mathcal{I}}$ and $\mathcal{B}_{\mathcal{I}}$ denote the corresponding reduced products over $\mathcal{I}$. Next proposition shows that if  each $\mathcal{A}_{\gamma}\equiv \mathcal{B}_{\gamma}$ for $\gamma\in \Omega$ then $\mathcal{A}_{\mathcal{I}}$ and $\mathcal{B}_{\mathcal{I}}$ are also elementarily equivalent.
\begin{proposition}\label{preservation}
Reduced products, direct products and ultraproducts preserve elementary equivalence.
\end{proposition}
\begin{proof}
We only need to show this for reduced products, since the others are special cases of reduced products. Let $ \mathcal{A}_{\mathcal{I}}$ and $ \mathcal{B}_{\mathcal{I}}$ be two reduced products over ideal $\mathcal{I}$ such that $\mathcal{A}_{\gamma}\equiv \mathcal{B}_{\gamma}$ for every $\gamma\in \Omega$. Let $\varphi$ be an $\mathcal{L}$-sentence. By Proposition \ref{F-restricted} we can assume $\varphi$ is an $\mathcal{F}_{0}$-restricted $\mathcal{L}$-sentence.  For a given $n\in \mathbb{N}$ suppose $\varphi$ is determined up to $2^{-n}$ by $( \sigma_{0},\dots,\sigma_{2^{n}}; \psi_{0}, \dots, \psi_{m-1})$. For each $i\in \{0, \dots, 2^{n}\}$ and $j\in \{0,\dots, m-1\}$ let
 \begin{equation}
 \nonumber X_{i}^{j}= \{\gamma\in \Omega : \psi_{j}^{\mathcal{A}_{\gamma}}> i/2^{n}\},
 \end{equation}
  \begin{equation}
 \nonumber Y_{i}^{j}= \{\gamma\in \Omega : \psi_{j}^{\mathcal{B}_{\gamma}}> i/2^{n}\}.
 \end{equation}
By our assumption $X_{i}^{j}=Y_{i}^{j}$ for all $i$ and $j$. Therefore
 \begin{equation}
 \nonumber P(\Omega)/\mathcal{I} \vDash \sigma_{i}([X^{0}_{0}]_{\mathcal{I}},\dots, [X^{m-1}_{2^{n}}]_{\mathcal{I}})\leftrightarrow \sigma_{i}([Y^{0}_{0}]_{\mathcal{I}},\dots, [Y^{m-1}_{2^{n}}]_{\mathcal{I}}),
 \end{equation}
 and 
 \begin{equation}
 \nonumber P(\Omega)/\mathcal{I} \vDash \sigma_{i}([\tilde{X}^{0}_{0}]_{\mathcal{I}},\dots, [\tilde{X}^{m-1}_{2^{n}}]_{\mathcal{I}})\leftrightarrow \sigma_{i}([\tilde{Y}^{0}_{0}]_{\mathcal{I}},\dots, [\tilde{Y}^{m-1}_{2^{n}}]_{\mathcal{I}}).
 \end{equation}
 If $\varphi^{\mathcal{A}_{\mathcal{I}}}> \ell/2^{n}$  by Theorem \ref{FV} and Definition \ref{det} (3) we have
  \begin{equation}
 \nonumber P(\Omega)/\mathcal{I} \vDash \sigma_{\ell}([\tilde{X}^{0}_{0}]_{\mathcal{I}},\dots, [\tilde{X}^{m-1}_{2^{n}}]_{\mathcal{I}})
  \end{equation}
 which is of course
   \begin{equation}
 \nonumber P(\Omega)/\mathcal{I} \vDash  \sigma_{\ell}([\tilde{Y}^{0}_{0}]_{\mathcal{I}},\dots, [\tilde{Y}^{m-1}_{2^{n}}]_{\mathcal{I}}).
 \end{equation}
 Thus Lemma \ref{weaker lemma} implies that
 $ \varphi^{\mathcal{B}_{\mathcal{I}}}>( \ell-1)/2^{n}$. Similarly $ \varphi^{\mathcal{B}_{\mathcal{I}}}> \ell/2^{n}$ implies that $ \varphi^{\mathcal{A}_{\mathcal{I}}}> (\ell-1)/2^{n}$.
This means that $|\varphi^{\mathcal{A}_{\mathcal{I}}} - \varphi^{\mathcal{B}_{\mathcal{I}}}|\leq 1/2^n$.
  Since $n$ was arbitrary by approaching $n$ to infinity, we have $\varphi^{\mathcal{A}_{\mathcal{I}}}=\varphi^{\mathcal{B}_{\mathcal{I}}}$.
Therefore
  $\mathcal{A}_{\mathcal{I}}\equiv \mathcal{B}_{\mathcal{I}}$.
\end{proof}

\begin{theorem}\label{atomless}
Assume $\mathcal{A}$ is an $\mathcal{L}$-structure and $\mathcal{I}$ and $\mathcal{J}$ are atomless ideals on $\Omega$, then the reduced powers of $\mathcal{A}$ over $\mathcal{I}$ and $\mathcal{J}$ are elementarily equivalent.
\end{theorem}
\begin{proof}

Let $\mathcal{A}_{\mathcal{I}}$ and $\mathcal{A}_{\mathcal{J}}$ denote the reduced powers of $\mathcal{A}$ ($\mathcal{A}_{\gamma}=\mathcal{A}$ for all $\gamma\in \Omega$) over $\mathcal{I}$ and $\mathcal{J}$, respectively. Let $\varphi$ be an $\mathcal{L}$-sentence and for $n\geq 1$ find an $\mathcal{F}_{0}$-restricted $\mathcal{L}$-sentence $\tilde{\varphi}$ which is uniformly within $2^{-n}$ of $\varphi$ and is determined up to $2^{-n}$ by $( \sigma_{0},\dots,\sigma_{2^{n}}; \psi_{0}, \dots, \psi_{m-1})$. Then
\begin{equation}
\nonumber X^{j}_{i}=\{\gamma\in\Omega: \psi_{j}^{\mathcal{A}_{\gamma}}> i/2^{n}\}
\end{equation}
is clearly either $\Omega$ or $\emptyset$, therefore $[X^{j}_{i}]_{\mathcal{I}}=[X^{j}_{i}]_{\mathcal{J}}= 0$ or $1$. Since any two atomless Boolean algebras are elementarily equivalent, for every $i=0,\dots , 2^{n}$
\begin{equation}
\nonumber  P(\Omega)/\mathcal{I} \vDash\sigma_{i}([X^{0}_{0}]_{\mathcal{I}}, \dots, [X^{m-1}_{2^{n}}]_{\mathcal{I}}) \Longleftrightarrow P(\Omega)/\mathcal{J} \vDash\sigma_{i}([X^{0}_{0}]_{\mathcal{J}}, \dots, [X^{m-1}_{2^{n}}]_{\mathcal{J}}).
\end{equation}
Thus by Theorem \ref{FV} and the same argument as the proof of Proposition \ref{preservation} we have $\varphi^{\mathcal{A}_{\mathcal{I}}}=\varphi^{\mathcal{B}_{\mathcal{I}}}$.
\end{proof}

\section{Isomorphisms of reduced products under the Continuum Hypothesis}\label{Last}

  In C*-algebra context an important class of corona algebras is the reduced power of a C*-algebra $\mathcal{A}$ over the Fr\'{e}chet ideal $Fin$. It is called the \emph{asymptotic sequence algebra} of $\mathcal{A}$ and denoted by $\ell_{\infty}(\mathcal{A})/c_{0}(\mathcal{A})$. The C*-algebra $\mathcal{A}$ can be identified with its diagonal image in $\ell_{\infty}(\mathcal{A})/c_{0}(\mathcal{A})$. We will also use the same notation $\ell_{\infty}(\mathcal{A})/c_{0}(\mathcal{A})$ for the reduced power of an arbitrary metric structure $\mathcal{A}$ over $Fin$.

  As mentioned in the introduction a result of Farah-Shelah (Theorem \ref{countable saturation}) shows that asymptotic sequence algebras are countably saturated and therefore if $\mathcal{A}$ is separable, assuming the Continuum Hypothesis, they have $2^{\aleph_{1}}$ automorphisms, hence non-trivial ones. The last statement follows from a folklore theorem (refer to \cite[Corollary 4.1]{FaShRig} for a proof due to Bradd Hart) which states that any $\kappa$-saturated metric structure of density character $\kappa$ has $2^{\kappa}$ automorphisms. Furthermore, $\kappa$-saturated structures of the same density character $ \kappa$ which are elementarily equivalent are isomorphic. Assuming  the Continuum Hypothesis,  for a separable $\mathcal{A}$ and a nonprincipal ultrafilter $\mathcal U$ on $\mathbb{N}$, the asymptotic sequence algebra $\ell_{\infty}(\mathcal{A})/c_{0}(\mathcal{A})$  and its ultrapower $\big(\ell_{\infty}(\mathcal{A})/c_{0}(\mathcal{A})\big)_\mathcal U$ have the same density character $\aleph_1$ (see \cite[Theorem 2.7]{multiplier}), and therefore they are isomorphic. We will show that this is also the case for the reduced powers of separable metric structures over a large family of ideals (Corollary \ref{coro}), e.g., $\ell_{\infty}(\mathcal{A})/c_{0}(\mathcal{A})$ is isomorphic to the asymptotic sequence algebra of $\ell_{\infty}(\mathcal{A})/c_{0}(\mathcal{A})$, under the Continuum Hypothesis.

We briefly recall a few definitions and facts about ideals on $\mathbb{N}$. The properties of various \emph{definable} ($F_{\sigma}$, $F_{\delta\sigma}$, Borel, analytic, $\dots$) ideals and their Boolean algebra quotients have been vastly studied, see e.g., \cite{FaAn}, \cite{Sol}.

For any $A\subseteq \mathbb{N}\times \mathbb{N}$ the \emph{vertical section} of $A$ at $m$ is the set $A_{m}=\{n\in \mathbb{N}: (m,n)\in A\}$.
  For two ideals $\mathcal{I}$ and $\mathcal{J}$ on $\mathbb{N}$, the \emph{Fubini product}, $\mathcal{I}\times \mathcal{J}$, of $\mathcal{I}$ and $\mathcal{J}$ is the ideal on $\mathbb{N}\times \mathbb{N}$ defined by
\begin{equation}
\nonumber A\in \mathcal{I}\times\mathcal{J} ~~\leftrightarrow~~ \{i: A_{i}\notin \mathcal{J}\}\in \mathcal{I}.
\end{equation}

\begin{definition}
A map $\mu : \mathcal{P}(\mathbb{N})\rightarrow [0,\infty]$ is a submeasure supported by $\mathbb{N}$ if for $A, B\subseteq \mathbb{N}$
\begin{eqnarray}
 \nonumber &\mu(\emptyset)=0\\
 \nonumber &\mu(A)\leq \mu(A\cup B)\leq\mu(A)+\mu(B).
\end{eqnarray}
It is lower semicontinuous if for all $A\subseteq \mathbb{N}$ we have
\begin{equation}
\nonumber \mu(A)= \lim_{n\rightarrow\infty}\mu(A\cap [1,n]).
\end{equation}
\end{definition}

\textbf{Layered ideals.} An ideal $\mathcal{I}$ is \emph{layered} if there is $f: P(\mathbb{N})\rightarrow [0, \infty]$ such that
\begin{enumerate}
\item $f(A)\leq f(B)$ if $A\subseteq B$,
\item $\mathcal{I}=\{A: ~ f(A)<\infty\}$,
\item $f(A)=\infty$ implies $f(A)=\sup_{B\subseteq A} f(B)$.
\end{enumerate}
Layered ideals were introduced in \cite{How}, where in particular the following is proved.
\begin{lemma}\label{layered}\cite[Proposition 6.6]{How}
\begin{enumerate}
\item Every $F_{\sigma}$ ideal is layered.
\item If $\mathcal{J}$ is a layered ideal and $\mathcal{I}$ is an arbitrary ideal on $\mathbb{N}$, then $\mathcal{J}\times\mathcal{I}$ is
layered.
\end{enumerate}
\end{lemma}
\begin{proof}
 By a theorem of K. Mazur (\cite{Mazur}) for every $F_{\sigma}$ ideal $\mathcal{I}$ there is a lower semicontinuous submeasure $\mu$ such that
 \begin{eqnarray}
\nonumber \mathcal{I}=\{A\subseteq \mathbb{N}~:~ \mu(A)< \infty\}
\end{eqnarray}
and $f=\mu$ satisfies all the conditions above. For (2) let $f_{\mathcal{J}}$ be a map witnessing that $\mathcal{J}$ is layered, and define $f$ by
 \begin{equation}
 \nonumber f(A)= f_{\mathcal{J}}(\{n: A_{n}\notin \mathcal{I}\})
 \end{equation}
 for $A\subseteq \mathbb{N}^{2}$. It is not hard to see that $f$ satisfies the conditions (1) - (3) stated above.
 \end{proof}

\begin{theorem}[Farah-Shelah]\label{countable saturation}
Every reduced product $\prod_{n}\mathcal{A}_{n}/ \bigoplus_{\mathcal{I}}\mathcal{A}_{n}$ is countably saturated if
 $\mathcal{I}$ is a layered ideal.
\end{theorem}
\begin{proof}
See \cite[Theorem 2.7]{FaShRig}.
\end{proof}
Therefore an immediate consequence of Theorem \ref{ee} and Theorem \ref{countable saturation} implies the following corollary.
\begin{corollary}\label{CH}
Assume the Continuum Hypothesis. If $\mathcal{A}$ is a separable metric structure, $\mathcal{I}$ and $\mathcal{J}$ are atomless layered ideals, then the reduced powers $\mathcal{A}_{\mathcal{I}}$ and $\mathcal{A}_{\mathcal{J}}$ are isomorphic.
\end{corollary}
In corollary \ref{coro} we give an application of this result, but before we need the following lemma.
\begin{lemma}\label{Fubini}
Suppose $\mathcal{I}$ and $\mathcal{J}$ are ideals on $\mathbb{N}$ and $\mathcal{A}_{\mathcal{I}}$ is the reduced power of $\mathcal{A}$ over the ideal $\mathcal{I}$. Then
\begin{equation}
\nonumber \frac{\prod\mathcal{A}_{\mathcal{I}}}{\bigoplus_{\mathcal{J}}\mathcal{A}_{\mathcal{I}}}\cong \frac{\prod_{\mathbb{N}^{2}}\mathcal{A}}{\bigoplus_{\mathcal{J}\times \mathcal{I}}\mathcal{A}}.
\end{equation}
\end{lemma}
\begin{proof}
Assume $\langle a_{n,m}\rangle$ is an element of $\prod_{\mathbb{N}^{2}}\mathcal{A}$.
Define the map $\rho: \prod_{\mathbb{N}^{2}}\mathcal{A}/\bigoplus_{\mathcal{J}\times \mathcal{I}}\mathcal{A} \rightarrow \prod(\mathcal{A}_{\mathcal{I}})/\bigoplus_{\mathcal{J}}(\mathcal{A}_{\mathcal{I}})$ by
\begin{equation}
\nonumber \rho(\pi_{_{\mathcal{J}\times \mathcal{I}}}(\langle a_{m,n}\rangle))= \pi_{_{\mathcal{J}}}(\langle b_{m}\rangle),
\end{equation}
 where $b_{m}=\pi_{_{\mathcal{I}}}(\langle a_{m,n}\rangle_{n})$ for each $m\in\mathbb{N}$.
In order to see this map is well-defined assume $\pi_{_{\mathcal{J}\times \mathcal{I}}}(\langle a_{m,n}\rangle)=0$. If $\pi_{_{\mathcal{J}}}(\langle b_{m}\rangle)\neq 0$, then there is $\epsilon>0$ such that for every $S\in \mathcal{J}$ we have
\begin{equation}
\nonumber \sup_{m\notin S}\|b_{m}\|_{\mathcal{A}_{\mathcal{I}}}\geq \epsilon.
\end{equation}
Since $\pi_{_{\mathcal{J}\times \mathcal{I}}}(\langle a_{n,m}\rangle)=0$, there is $X\in{\mathcal{J}\times \mathcal{I}}$ such that
\begin{equation}
\nonumber \sup_{(m,n)\notin X}\|\langle a_{m,n}\rangle)\|_{\mathcal{A}}< \epsilon/4.
\end{equation}
 The set $S=\{m: X_{m}\notin \mathcal{I}\}$ belongs to $\mathcal{J}$ and hence $\sup_{m\notin S}\|b_{m}\|_{\mathcal{A}_{\mathcal{I}}}\geq \epsilon$. Pick $m_{0}\notin S$ such that $\|b_{m_{0}}\|_{\mathcal{A}_{\mathcal{I}}}\geq \epsilon/2$ and then pick $n_{0}\notin X_{m_{0}}$ such that $\|a_{m_{0},n_{0}}\|_{\mathcal{A}}\geq \epsilon/4$, which is a contradiction. Therefore $\pi_{_{\mathcal{J}}}(\langle b_{m}\rangle)=0$.

To show the injectivity of $\rho$ assume $\pi_{_{\mathcal{J}}}(\langle b_{m}\rangle)=0$. Therefore for every $\epsilon >0$  there is $S\in \mathcal{J}$ such that $\|b_{m}\|_{\mathcal{A}_{\mathcal{I}}}\leq\epsilon$ for every $m\in \mathbb{N}\setminus S$. So for each $m\in \mathbb{N}\setminus S$ there is $X_{m}\in\mathcal{I}$ such that
\begin{equation}
\nonumber \sup_{n\notin X_{m}} \|a_{(m,n)}\|_{\mathcal{A}}\leq 2\epsilon.
\end{equation}
The set $X=(S\times \mathbb{N})\cup\{(m,n): n\in X_{m}\}$ belongs to the ideal $\mathcal{J}\times \mathcal{I}$ and
\begin{equation}
\nonumber \sup_{(m,n)\notin X} \|a_{m,n}\|_{\mathcal{A}}\leq 2\epsilon.
\end{equation}
Therefore $\pi_{_{\mathcal{J}\times \mathcal{I}}}(\langle a_{n,m}\rangle)=0$. It is easy to check that $\rho$ is a surjective *-homomorphism.
\end{proof}
The following corollary follows form Lemma \ref{layered} and Corollary \ref{CH}.
\begin{corollary}\label{coro}
Assume the Continuum Hypothesis. Suppose $\mathfrak{A}=\ell_{\infty}(\mathcal{A})/c_{0}(\mathcal{A})$ is the asymptotic sequence algebra of $\mathcal{A}$ and $\mathcal{I}$ is an atomless layered ideal on $\mathbb{N}$, then
\begin{equation}
\nonumber  \frac{\prod\mathfrak{A}}{\bigoplus_{\mathcal{I}} \mathfrak{A}}\cong \mathfrak{A}.
\end{equation}
\end{corollary}

\section{Non-trivially isomorphic reduced products of matrix algebras.}\label{Matrices}
 In this section we use Theorem \ref{FV} in order to prove the existence of two reduced products of matrix algebras which are isomorphic under the Continuum Hypothesis, but not isomorphic in ZFC. Note that in  model theory for operator algebras, the ranges of formulas are bounded subsets of reals possibly different from $[0,1]$ (see for example \cite{FHS}). Nevertheless Definition \ref{det} can be easily adjusted for any formula in the language of C*-algebras $\mathcal{L}$ and Theorem \ref{FV} can be proved similarly.

  As mentioned in the introduction, commutative examples of such reduced products are well-known, for example by a classical result of Parovi\v{c}enko, under the Continuum Hypothesis $(\ell_{\infty}(\mathbb{N})/c_{0}(\mathbb{N})\cong) C(\beta\omega\setminus \omega)\cong C(\beta\omega^{2}\setminus \omega^{2})$, however under the proper forcing axiom they are not isomorphic, since there are no trivial isomorphisms between them (see \cite[Chapter 4]{FaAn}). Other examples of non-trivial isomorphisms between (non-commutative) reduced products can be obtained by tensoring a matrix algebra with these commutative algebras. Recall that (\cite{Black}) for a locally compact Hausdorff topological space $X$ and for any C*-
algebra $\mathcal{A}$, $C_{0}(X,\mathcal{A})$ can be identified with $C_{0}(X)\otimes \mathcal{A}$, under the map $(f\otimes a)(x) = f(x)a$. Let $M_{n}$ denote the space of all $n\times n$ matrices over the field of complex numbers.  Assume $\mathfrak{A}=\ell_{\infty}(M_{2})/c_{0}(M_{2})$  is the asymptotic sequence algebra of $M_{2}$, we have
\begin{equation}
\nonumber \mathfrak{A}\cong C(\beta\omega\setminus \omega)\otimes M_{2}\cong M_{2}(C(\beta\omega\setminus\omega)),
\end{equation}
  and Corollary \ref{coro} implies that
 \begin{equation}
 \nonumber \frac{\ell_{\infty}(\mathfrak{A})}{c_{0}(\mathfrak{A})}\cong \frac{\prod_{\mathbb{N}^{2}} M_{2}}{\bigoplus_{Fin\times Fin} M_{2}}\equiv \mathfrak{A}.
 \end{equation}
Since $\prod_{\mathbb{N}^{2}} M_{2}/\bigoplus_{Fin\times Fin} M_{2}\cong M_{2}(C(\beta\omega^{2}\setminus \omega^{2}))$,
 for the same reason as the commutative case, under the Continuum Hypothesis $\mathfrak{A}$ and $\ell_{\infty}(\mathfrak{A})/c_{0}(\mathfrak{A})$ are non-trivially isomorphic.

Recall that an ideal $\mathcal{I}$ on $\mathbb{N}$ is a P-ideal if for every sequence $\{A_{n}\}_{n=1}^{\infty}$ of sets in $\mathcal{I}$ there is a  set $A_{\infty}\in \mathcal{I}$ such that $A_{n}\setminus A_{\infty}$ is finite, for all $n$.

\begin{theorem}\label{G}\cite[Corollary 1.1]{Ghasemi} Assume there is a measurable cardinal. 
It is relatively consistent with ZFC that for every analytic P-ideal $\mathcal{I}$ on $\mathbb{N}$, the reduced products $\prod M_{k(n)}/\bigoplus_{\mathcal{I}} M_{k(n)}$ and $\prod M_{l(n)}/\bigoplus_{\mathcal{I}} M_{l(n)}$ are isomorphic if and only if there are sets $A,B\in \mathcal{I}$ and a bijection $\nu: \mathbb{N}\setminus A \rightarrow \mathbb{N}\setminus B$ such that $k(\nu(n))=l(n)$ for all $n\in \mathbb{N}\setminus A$.

Moreover if $\Phi$ is such an isomorphism, there exists a bounded linear operator $u:\prod_{n\in\mathbb{N}} M_{k(n)}\rightarrow \prod_{n\in\mathbb{N}} M_{l(n)}$ such that $\pi_{_{\mathcal{I}}}(u)$ is a unitary and $\Phi(\pi_{_{\mathcal{I}}}(a))=\pi_{_{\mathcal{I}}}(uau^{*})$.
 \end{theorem}
 Such isomorphisms are  trivial in the sense of \cite{Trivial} and in the sense of Definition \ref{trivial-definition} if $\mathcal I= \mathcal J = Fin$.

\begin{lemma}
There is an increasing sequence of natural numbers $\{k_{\infty}(i): i\in\mathbb{N}\}$ such that for every $\mathcal{F}_{0}$-restricted $\mathcal{L}$-sentence $\psi$
 \begin{equation}
 \nonumber \lim_{i} \psi^{M_{k_{\infty}(i)}}=r_{\psi}
 \end{equation}
 for some real number $r_{\psi}$.
\end{lemma}
\begin{proof}
Let $\psi_{1}, \psi_{2}, \dots $ be an enumeration of all $\mathcal{F}_{0}$-restricted $\mathcal{L}$-sentences. Starting with $\psi_{1}$, since the range of it is a bounded set, find a sequence $\{k_{1}(i)\}$ such that $ \psi_{1}^{M_{k_{1}(i)}} \rightarrow r_{\psi_{1}}$ for some $r_{\psi_{1}}$. Similarly find a subsequence $\{k_{2}(i)\}$ of $\{k_{1}(i)\}$ such that $ \psi_{2}^{M_{k_{2}(i)}} \rightarrow r_{\psi_{2}}$ for some $r_{\psi_{2}}$, and so on. If we let
\begin{equation}
\nonumber k_{\infty}(i)= k_{i}(i) \qquad\qquad i\in \mathbb{N},
\end{equation}
then $\{k_{\infty}(i)\}$ has the required property.
\end{proof}

\begin{proposition}\label{nontrivial}
For any ideal $\mathcal{I}$ on $\mathbb{N}$ containing all finite sets, if $\{g(i)\}$ and $\{h(i)\}$ are two almost disjoint subsequences of $\{k_{\infty}(i)\}$, then
\begin{equation}
\nonumber \frac{\prod_{i}M_{g(i)}}{\bigoplus_{\mathcal{I}}M_{g(i)}}\equiv \frac{\prod_{i} M_{h(i)}}{\bigoplus_{\mathcal{I}}M_{h_(i)}},
\end{equation}
hence if $\mathcal{I}$ is a layered P-ideal, they are isomorphic under the Continuum Hypothesis, with no trivial isomorphisms between them.
\end{proposition}
This together with Theorem \ref{G} implies that these reduced products are not isomorphic in ZFC, and therefore Theorem \ref{main} follows.
\begin{proof}
 Let $\varphi$ be an $\mathcal{L}$-sentence and for $n\geq 1$ find an $\mathcal{F}_{0}$-restricted $\mathcal{L}$-sentence $\tilde{\varphi}$ which is uniformly within $2^{-n}$ of $\varphi$ and it is determined up to $2^{-n}$ by $(\sigma_{0},\dots,\sigma_{2^{n}}; \psi_{0}, \dots,  \psi_{m-1})$. Let
\begin{equation}
\nonumber X^{j}_{i}=\{l\in\mathbb{N}: \psi_{j}^{M_{g(l)}}> i/2^{n}\}
\end{equation}
and
\begin{equation}
\nonumber Y^{j}_{i}=\{l\in\mathbb{N}: \psi_{j}^{M_{h(l)}}> i/2^{n}\}.
\end{equation}
Since each $\psi_{j}$ is an $\mathcal{F}_{0}$-restricted formula and $\lim_{l} \psi_{j}^{M_{k_{\infty}(l)}}=r_{\psi_{j}}$ and $\mathcal{I}$ contains all finite sets, we have $[X^{j}_{i}]_{\mathcal{I}}=[Y^{j}_{i}]_{\mathcal{I}}$. Hence Theorem \ref{FV} implies that $\tilde{\varphi}^{\mathcal{A}_{\mathcal{I}}}=\tilde{\varphi}^{\mathcal{B}_{\mathcal{I}}}$. By uniform density of $\mathcal{F}_{0}$-restricted $\mathcal{L}$-sentences, the result follows.
\end{proof}

The following theorem shows the abundance of different theories of reduced products of matrix algebras.
\begin{theorem}\label{many}
For any ideal $\mathcal{I}$, there are $2^{\aleph_{0}}$-many reduced products of matrix algebras over $\mathcal{I}$ which are pairwise non-elementarily equivalent.
\end{theorem}
\begin{proof}
Let $E=\{p_{1}, p_{2}, \dots\}\subset \mathbb{N}$ be an increasing enumeration of prime numbers. Assume $\{A_{\xi}: \xi<2^{\aleph_{0}}\}$ is an almost disjoint family of subsets of $E$. Let $A_{\xi}=\{n_{1}^{\xi}, n^{\xi}_{2}, \dots\}$ be an increasing enumeration of $A_{\xi}$. For each $\xi < 2^{\aleph_{0}}$ define a sequence $\langle k^{\xi}(n)\rangle$ of natural numbers by
\begin{equation}
\nonumber \langle k^{\xi}(n)\rangle=\langle n_{1}^{\xi}, n_{1}^{\xi}n_{2}^{\xi}, n_{1}^{\xi}n_{2}^{\xi}n_{3}^{\xi}, \dots\rangle.
\end{equation}
We will show that for any distinct $\xi, \eta <2^{\aleph_{0}}$ the reduced products $\prod M_{k^{\xi}(n)}/\bigoplus_{\mathcal{I}}M_{k^{\xi}(n)}$ and $\prod M_{k^{\eta}(n)}/\bigoplus_{\mathcal{I}}M_{k^{\eta}(n)}$ are not elementarily equivalent.

Fix such $\xi$ and $\eta$. Since $A_{\xi}$ and $A_{\eta}$  are almost disjoint, pick $m$ such that $\{n_{m}^{\xi}, n^{\xi}_{m+1}, \dots\}\cap \{n_{1}^{\eta}, n^{\eta}_{2}, \dots\}= \emptyset$. Define a formula $\varphi(\bar{x}, \bar{y})$ by
\begin{eqnarray}
 \varphi(x_1, \dots, x_{n^{\xi}_{m}}, y_{2}, \dots, y_{n^{\xi}_{m}})&=& \sum_{i=1}^{n^{\xi}_{m}}(\|x_{i}^{2}-x_{i}\|+\|x_{i}^{*}-x_{i}\|)+\|\sum_{i=1}^{n^{\xi}_{m}} x_{i}-1\| \nonumber\\
 && +\sum_{i\neq j}^{n^{\xi}_{m}}\|x_{i}x_{j}\|+\sum_{i=2}^{n^{\xi}_{m}}(\|y_{i}y_{i}^{*}- x_{1}\|+ \|y_{i}^{*}y_{i}- x_{i}\|).\nonumber
\end{eqnarray}
For a unital C*-algebra $\mathcal{A}$ and tuples $\bar{a}$ and $\bar{v}$ if $\varphi(\bar{a}, \bar{v})^{\mathcal{A}}=0$ then $a_{1}, a_{2}, \dots, a_{n^{\xi}_{m}}$ are orthogonal pairwise Murray-von Neumann equivalent projections of $\mathcal{A}$, and therefore $M_{n^{\xi}_{m}}$ can be embedded into $\mathcal{A}$.
Since for every $k\geq m$
\begin{equation}
\nonumber n^{\xi}_{m}\mid n^{\xi}_{1}n^{\xi}_{2}\dots n^{\xi}_{k},
\end{equation}
one can easily find a tuple of projections $\bar{a}$ and a tuple of partial isometries $\bar{v}$ in $\prod M_{k^{\xi}(n)}$ such that
\begin{equation}
\nonumber \varphi(\pi_{_{\mathcal{I}}}(\bar{a}), \pi_{_{\mathcal{I}}}(\bar{v}))^{\prod M_{k^{\xi}(n)}/\bigoplus_{\mathcal{I}}M_{k^{\xi}(n)}}=0,
\end{equation}
but it is not possible to find such projections and partial isometries in $\prod M_{k^{\eta}(n)}$ since
\begin{equation}
\nonumber n^{\xi}_{m}\nmid n^{\eta}_{1}n^{\eta}_{2}\dots n^{\eta}_{k}
\end{equation}
for every $k\in \mathbb{N}$. Hence $\prod M_{k^{\xi}(n)}/\bigoplus_{\mathcal{I}}M_{k^{\xi}(n)} \not\equiv\prod M_{k^{\eta}(n)}/\bigoplus_{\mathcal{I}}M_{k^{\eta}(n)}$.
\end{proof}

\section{further remarks and questions}
For a locally compact Hausdorff topological space $X$ and a metric structure $\mathcal{A}$ the \emph{continuous reduced products} $C_{b}(X, \mathcal{A})/C_{0}(X,\mathcal{A})$ are studied as models for metric structures in \cite{FaShRig}, where in particular it has been shown that certain continuous reduced products, e.g., $C([0,1)^{*})$, are \emph{countably saturated}. In general $C_{b}(X, \mathcal{A})$ is a submodel of $\prod_{t\in X} \mathcal{A}$ and
 one may hope to use a similar approach as in section 4 in order to prove the following preservation (of $\equiv$) question.

\textbf{Question (1).} Assume $\mathcal{A}$ and $\mathcal{B}$ are elementarily equivalent metric structures and $X$ is a locally compact, non-compact Polish space. Are  $C_{b}(X, \mathcal{A})/C_{0}(X,\mathcal{A})$ and $C_{b}(X, \mathcal{B})/C_{0}(X,\mathcal{B})$  elementarily equivalent?

Note that if $X$ is a discrete space (e.g., $\mathbb{N}$) this follows from Proposition \ref{preservation} since $C_{b}(X, \mathcal{A})/C_{0}(X,\mathcal{A})\cong \prod_{t\in X}\mathcal{A}/\bigoplus \mathcal{A}$.

 In \cite{7authors} the authors showed the existence of two C*-algebras $\mathcal{A}$ and $\mathcal{B}$ such that $\mathcal{A}\equiv \mathcal{B}$, where $C([0,1])\otimes \mathcal{A} \not\equiv C([0,1])\otimes \mathcal{B}$, i.e., tensor products in the category of C*-algebras, do not preserve elementary equivalence. In general it is not clear ``how the theory of $C_{0}(X,\mathcal{A})$ is related to the theory of $\mathcal{A}$".

\textbf{Question (2).} Assume $\mathcal{A}$ and $\mathcal{B}$ are elementarily equivalent C*-algebras. For which locally compact, Hausdorff spaces, like $X$,  $C_{0}(X)\otimes \mathcal{A}\equiv C_{0}(X)\otimes \mathcal{B}$ is true?

We conclude with an interesting observation which might give some insight to the previous question.

\textbf{ The Cone and Suspension Algebras.}  Let $\mathcal{A}$ be a C*-algebra. The \emph{cone} $C\mathcal{A}= C_{0}((0,1], \mathcal{A})$ and \emph{suspension} $S\mathcal{A}=C_{0}((0,1), \mathcal{A})$  over $\mathcal{A}$ are the most important examples of contractible and subcontractible C*-algebras (\cite{Black}). Since $S\mathcal{A}\subset C\mathcal{A}$ and $C\mathcal{A}$ is homotopic to $\{0\}$ (contractible) by a well-known result of D. Voiculescu (\cite{Voi}) both $C\mathcal{A}$ and $S\mathcal{A}$ are \emph{quasidiagonal} C*-algebras. Every quasidiagonal C*-algebra embeds into a reduced product of full matrix algebras over the Fr\'{e}chet ideal, $\prod M_{k(n)}/\bigoplus_{\mathcal{I}} M_{k(n)}$ for some sequence $\{k(n)\}$ (such C*-algebras are called MF, see e.g., \cite{Brown} and \cite{Black}). In general it is easy to check that if a metric structure $\mathcal{A}$ embeds into $\mathcal{B}$ then the universal theory of $\mathcal{A}$, $Th_{\forall}(\mathcal{A})$, contains the universal theory of $\mathcal{B}$, $Th_{\forall}(\mathcal{B})$. Hence for any C*-algebra $\mathcal{A}$

\begin{equation}
\nonumber  C\mathcal{A} \hookrightarrow \prod M_{k(n)}/\bigoplus M_{k(n)}
\end{equation}
for some $\{k(n)\}$, which implies that

\begin{equation}
\nonumber Th_{\forall}(C\mathcal{A}) \supseteq  Th_{\forall}(\prod M_{k(n)}/\bigoplus M_{k(n)}).
\end{equation}


\begin{thebibliography}{100}
\bibitem{multiplier} C. A. Akemann,  G. K. Pedersen, J. Tomiyama,  Multipliers of C*-algebras. Journal of Functional Analysis, 13(3), 277-301, 1973.
\bibitem{Ben} I. Ben Yaacov, A. Berenstein, C.W. Henson, and A. Usvyatsov, Model theory
for metric structures, Model Theory with Applications to Algebra and Analysis, vol. II (Z. Chatzidakis et al., eds.), London Math. Soc. Lecture Notes
Series, no. 350, Cambridge University Press, pp. 315-427, 2008.
\bibitem{Black}  B. Blackadar,  Operator algebras, Encyclopaedia of Math. Science, vol. 122, Springer-Verlag, 2006.
\bibitem{Brown} N. P. Brown, On quasidiagonal C*-algebras. In Operator algebras
and applications,  Adv. Stud. Pure Math., Math. Soc. Japan, Tokyo, vol. 38, pp. 19–64, 2004.
\bibitem{Chang} C. C. Chang and H.J. Keisler, Model theory, third ed., Studies in Logic and the Foundations of Mathematics, North-Holland Publishing Co. Amsterdam, vol. 73, 1990.
\bibitem{DowHart} A. Dow and K.P. Hart, $\omega^{*}$ has (almost) no continuous images, Israel Journal of Mathematics, vol. 109, pp. 29-39, 1999.
\bibitem{FaAn} I. Farah, Analytic quotients: theory of liftings for quotients over analytic ideals on the integers, Memoirs of American Mathematical Society, vol. 147, no. 702, 2000.
\bibitem{FaCalkin} I. Farah, All automorphisms of the Calkin algebra are inner, Annals of Mathematics
173, 619-661, 2011.
 \bibitem{How} I. Farah, How many Boolean algebras $\mathcal{P}(\mathbb{N})/\mathcal{I}$ are there?, Illinois Journal of Mathematics, vol. 46 , pp. 999-1033, 2003.
  \bibitem{FaICM} I. Farah, Logic and operator algebras, Proceedings of the Seoul ICM, to appear.
 \bibitem{FH}  I, Farah and B. Hart, Countable saturation of corona algebras, C. R. Math. Rep. Acad. Sci. Canada, vol. 35 (2), pp. 35-56.  2013.
 \bibitem{7authors} I. Farah, B. Hart, M. Lupini, L. Robert, A. Tikuisis, A. Vignati and W. Winter, Model theory of nuclear C*-algebras, preprint.
\bibitem{FHS} I. Farah, B. Hart, and D. Sherman, Model theory of operator algebras II: Model theory, Israel Journal of Mathematics, 201, pp. 477-505, 2014.
\bibitem{Trivial}I. Farah, and S. Shelah, Trivial automorphisms, Israel Journal of Mathematics 201, no. 2, pp. 701-728, 2014.

\bibitem{FaShRig} I. Farah, S. Shelah,  Rigidity of continuous quotients, to appear in J. Math. Inst. Jussieu,  2014.
\bibitem{F-V} S. Feferman, R. L. Vaught, The first order properties of products of algebraic systems, Fundamenta Mathematicae, T. XLVII,  pp. 57-103, 1959.
\bibitem{FMS} T. Frayne, A. C. Morel and D. S. Scott, Reduced direct products, Fundamenta Mathematicae, vol. 51, pp. 195-228, 1962.
\bibitem{Ghasemi} S. Ghasemi, Isomorphisms of quotients of FDD-algebras , Israel Journal of Mathematics,Vol. 209, Issue 2, pp. 825-854, 2015.
\bibitem{Bradd} B. Hart, Continuous model theory and its applications, 2012, Course notes, available at
http:$/ /$www.math.mcmaster.ca/$\sim$bradd/courses/math712/index.html.
\bibitem{Lopes} V.C. Lopes, Reduced products and sheaves of metric structures, Math. Log. Quart. 59, No. 3, pp. 219-229, 2013.
\bibitem{Mazur} K. Mazur, $F_{\sigma}$-ideals and $\omega_{1}\omega_{1}^{*}$-gaps in the Boolean algebra $P(\omega)/\mathcal{I}$. Fundamenta Mathematicae,
138, pp. 103-111, 1991.
\bibitem{parov} I. I. Parovi\v{c}enko, A universal bicompact of weight $\aleph$, Soviet Mathematics Doklady 4, pp. 592-592, 1963; Russian original: Ob odnom universal'nom bikompakte vesa $\aleph$, Dokl. Akad. Nauk SSSR 150, pp. 36-39, 1963.
\bibitem{PhilWeav}N.C. Phillips and N. Weaver, The Calkin algebra has outer automorphisms, Duke
Math. Journal 139, 185-202, 2007.
\bibitem{Sol} S. Solecki, Analytic ideals and their applications. Annals of Pure and Applied Logic, vol. 99, pp. 51-72, 1999.
\bibitem{Voi} D. Voiculescu. A note on quasi-diagonal C*-algebras and homotopy, Duke Math. J., vol. 62(2), pp. 267-271, 1991.
\bibitem{Yiannis}Y. Vourtsanis, A direct droof of the Feferman-Vaught theorem and other preservation theorems in products,
The Journal of Symbolic Logic, vol. 56, no. 2 , pp. 632-636,  1991.
\end{thebibliography}
\end{document}